\documentclass[runningheads]{llncs}
\usepackage[T1]{fontenc}
\usepackage{amsmath,amssymb}

\usepackage{graphicx}

\usepackage{turnstile}
\usepackage{proof}
\usepackage{float}
\usepackage{hyperref}

\usepackage{color}

\spnewtheorem{fact}[theorem]{Fact}{\bfseries}{\itshape}

\spnewtheorem*{well-ordering_proof_CR}{Lemma~6}
{\bfseries}{\itshape}

\DeclareMathSymbol{\mhyph}{\mathalpha}{operators}{`-}

\newcommand{\deq}{\mathbin{\dot=}}
\newcommand{\dra}{\mathbin{\dot\rightarrow}}
\newcommand{\dfa}{\mathbin{\dot\forall}}

\newcommand*{\sub}{\mathrm{sub}}

\newcommand*{\num}[1]{\underline{#1}}

\newcommand*{\pred}[1]{\mathrm{#1}}

\newcommand*{\gn}[1]{\ulcorner{#1}\urcorner}
  
\newcommand\pole{{\protect\mathpalette{\protect\polehelper}{\bot}}} \def\polehelper#1#2{\mathrel{\rlap{$#1#2$}\mkern3mu{#1#2}}}

\newcommand{\Lt}{\mathcal{L}_{\mathrm{T}}}
\newcommand{\Lr}{\mathcal{L}_{\mathrm{R}}}
\newcommand{\T}{\mathit{T}}
\newcommand{\F}{\mathit{F}}

\begin{document}
%
%\title{Axiomatising Classical Realisability}
\title{A Friedman--Sheard-style Theory for Classical
Realisability}
%
%\titlerunning{Abbreviated paper title}
% If the paper title is too long for the running head, you can set
% an abbreviated paper title here
%
%\author{First Author\inst{1}\orcidID{0000-1111-2222-3333} \and
%Second Author\inst{2,3}\orcidID{1111-2222-3333-4444} \and
%Third Author\inst{3}\orcidID{2222--3333-4444-5555}}
\author{Daichi Hayashi\inst{1} and
Graham E. Leigh\inst{2}}
%
%\authorrunning{F. Author et al.}
% First names are abbreviated in the running head.
% If there are more than two authors, 'et al.' is used.
%
\institute{
\email{daichinhayashi0611@gmail.com} \and
Department of Philosophy, Linguistics and Theory of Science, University of Gothenburg, Göteborg, Sweden \\
\email{graham.leigh@gu.se}\\
}%\url{http://www.springer.com/gp/computer-science/lncs}}
%ABC Institute, Rupert-Karls-University Heidelberg, Heidelberg, Germany\\
%\email{\{abc,lncs\}@uni-heidelberg.de}}
%
\maketitle              % typeset the header of the contribution
\begin{abstract}
In Hayashi and Leigh (2024), the authors formulate classical number realisability for first-order arithmetic and a corresponding axiomatic system based on Krivine's classical realisability interpretation.
This paper presents a self-referential generalisation of previous results in the spirit of Friedman and Sheard (1987).
\end{abstract}
\section{Introduction}

Since Tarski's definition of truth for a formal language \cite{tarski1983concept}, various hierarchical or self-referential definitions of truth and their axiomatisations have been proposed (cf.~\cite{cantini1990theory,feferman1991reflecting,field2008saving,friedman1987axiomatic,halbach2014axiomatic,leigh2010ordinal,leitgeb2005truth}).
Truth is usually taken to be a property of truth bearers, such as sentences or propositions, and specific information on how such truth-bearers acquired the truth is not taken into account in those truth definitions.
On the contrary, Kleene \cite{kleene1945interpretation} formulated the notion of \emph{realisability} as a relation holding between an arithmetical formula and a \emph{realiser}, a natural number that represents a computational verification of the formula, and proved that each theorem of first-order Heyting arithmetic can be provably assigned a realiser and thus be justified.

While Kleene's interpretation targets intuitionistic theories,
Krivine successfully formulated realisability interpretation that can accommodate classical logic \cite{krivine2001typed,krivine2003dependent,krivine2009realizability}.
A novel feature of Krivine's realisability is that each formula is associated not only a set of \emph{realisers} but also a set of \emph{counter-realisers}, or \emph{refuters}, with
the couple of a refuter and realiser indicating a contradiction.
The set of all contradictory pairs is called a \emph{pole} and is fixed in advance, leading to different interpretations of realiser and refuter.
If there is an entity that realises a formula regardless of the choice of pole, then the formula is said to be \emph{universally} realisable.
Krivine proved that each theorem of classical second-order arithmetic is universally realisable.
As truth in the standard model coincides with realisability relative to the empty pole, universal realisability implies truth.
That is, universal realisability mediates provability and truth for classical logic.
Therefore, Krivine's interpretation indeed provides soundness of arithmetic, and Tarskian truth can be positioned in Krivine's general semantic framework. 

Just as the Tarskian truth definition can be axiomatised as a first-order theory $\mathsf{CT}$ (\emph{compositional truth}) over a weak arithmetic by a unary truth predicate, 
Krivine's classical realisability can be given an axiomatic counterpart $\mathsf{CR}$ (\emph{compositional realisability})~\cite{hayashi2024compositional}.
However, given that many other self-referential generalisations of $\mathsf{CT}$ have been proposed by authors, we can expect to formulate corresponding self-referential systems for classical realisability.
As a first step in this direction, this paper particularly aims at formulating a theory that corresponds to the Friedman--Sheard system \( \mathsf{FS} \) introduced in \cite{friedman1987axiomatic}.

The structure of this paper is as follows.
In the remainder of this section, we repeat the definition of classical number realisability in \cite{hayashi2024compositional} and its axiomatisation $\mathsf{CR}$, and then we restate some results on $\mathsf{CR}$.
In Section~\ref{ch:realisability_sec:FSR}, we define a system of Friedman--Sheard realisability $\mathsf{FSR}$ and its extensions.
Subsequently, we provide a revisional model and some proof-theoretic results for those systems.
In Section~\ref{sec:self-realisability}, we prove \emph{self realisability} of $\mathsf{FSR}$, that is, every theorem of $\mathsf{FSR}$ is formally realisable within $\mathsf{FSR}$.
In Section~\ref{sec:applications}, we will see two applications of self referentiality of $\mathsf{FSR}$: a proof-theoretic lower bound of extensions of $\mathsf{FSR}$ and an $\omega$-inconsistency-like phenomenon.  
In the final section, potential future work is discussed.

%%%%%%%%%%%%%%%%%%%%%%%%%%%%%%%%%%%%%%%%%%%%%%%%%%%%%%%%%%%

\subsection{Conventions and Notations}
We introduce the same abbreviations as used in \cite{hayashi2024compositional} for common formal concepts concerning coding and recursive functions. %, which are mainly based on \cite{leigh2013proof,leigh2016reflecting,leigh2010ordinal}.
We denote by \( \mathcal L \) the first-order language of $\mathsf{PA}$.
The logical symbols of $\mathcal{L}$ are $\to$, $\forall$ and \( = \).
The non-logical symbols are a constant symbol $\num 0$ and the function symbols for all primitive recursive functions. 
In particular, $\mathcal{L}$ has the successor function $x + 1,$ with which we can define \emph{numerals} $\num 0,\num 1,\num 2,\dotsc$. 
Thus, we identify natural numbers with the corresponding numeral.
We also employ the false equation $0=1$ as the propositional constant $\bot$ for contradiction, and then
the other logical symbols are defined in a standard manner, e.g., $\neg A = A \to \bot$.

The primitive recursive pairing function is denoted $\langle \cdot, \cdot \rangle$ with projection functions $(\cdot)_0$ and $(\cdot)_1$ satisfying $(\langle x,y \rangle)_0 = x$ and $(\langle x,y \rangle)_1 = y$. 
Sequences are treated as iterated pairing: $\langle x_0, x_1,\dots,x_k \rangle := \langle x_0, \langle x_1, \dots, x_k \rangle \rangle$.
Based on these constructors, each finite extension of \( \mathcal{L} \) is associated a fixed G\"{o}del coding (denoted \( \ulcorner e \urcorner \) where \( e \) is a finite string of symbols of the extended language) for which the basic syntactic constructions are primitive recursive.
In particular, \( \mathcal L \) contains a binary function symbol \( \sub \) representing the mapping \( \gn{A(x)} , n \mapsto \gn{A( n)} \) in the case that \( x \) is the only free variable of \( A(x) \), and binary function symbols \( \deq \), \( \dra \) and \( \dfa \) representing, respectively, the operations \( \gn s, \gn t \mapsto \gn{ s = t } \), \( \gn A, \gn B \mapsto \gn{ A \rightarrow B } \) and \( \gn x, \gn A \mapsto \gn{ \forall x A } \).
These operations will sometimes be omitted and we write \( \gn{s=t} \) and \( \gn{A\to B} \) for \( \deq( \gn s , \gn t ) \) and \( \dra(\gn A ,\gn B) \), etc.

We introduce a number of abbreviations for \( \mathcal L \)-expressions corresponding to common properties or operations on G\"odel codes.
The property of being the code of a variable is expressed by the formula \( \pred{Var}(x) \),
$\pred{ClTerm}(x)$ denotes the formula expressing that $x$ is a code of a \emph{closed} $\mathcal{L}$-term, and 
for a fixed extension \( \mathcal L' \) of \( \mathcal L \), $\pred{Sent}_{\mathcal{L}'}(x)$ expresses that $x$ is the code of a sentence.
The concepts above are primitive recursively definable meaning that the representing \( \mathcal L \)-formula is simply an equation between \( \mathcal L \)-terms.
Given a formula $A(x)$ with at most $x$ is free, $\gn {A(\dot{x})}$ abbreviates the term $\sub(\gn {A(x)} , x)$ expressing the code of the formula \( A(n) \) where \( n \) is the value of \( x \).
For a sequence $\vec{x}$ of variables, $\gn {A(\dot{\vec{x}})} $ is defined similarly.

Quantification over codes is associated similar abbreviations:
\begin{itemize}
	\item \( \forall \ulcorner A \urcorner \in \pred{Sent}_{\mathcal{L}'}. \ B( \ulcorner A \urcorner ) \) abbreviates \( \forall x (\pred{Sent}_{\mathcal{L}'}(x) \to B( x )) \).
	\item  \( \forall \gn s .\ B( \gn{s}  )  \) abbreviates
		\(
			\forall x( \pred{ClTerm}(x) \to B( x ) )
		\).
%	\item \( \forall \gn { A \rightarrow B} \in \pred{Sent}_{\mathcal{L}'}.\  B( \gn{ A \rightarrow B } , \gn A , \gn B ) \) abbreviates
%		\[
%		\forall x \forall y(\pred{Sent}_{\mathcal{L}'}(x \dra y) \to B( x \dra y , x , y ) ).
%		\]
	\item \( \forall \gn { A_v } \in \pred{Sent}_{\mathcal{L}'}.\ B( v , \gn{ A } ) \) abbr.
		\(
		\forall x\forall v(\pred{Var}(v) \wedge \pred{Sent}_{\mathcal{L}'}( \dfa v x ) \to B( v, x ) ),
		\) namely quantification relative to codes of formulas with at most one distinguished variable free.
\end{itemize}

Partial recursive functions can be expressed in \( \mathcal L \) via the Kleene `T predicate' method.
The ternary relation $ x \cdot y \simeq z $ expresses that the result of evaluating the \( x \)-th partial recursive function on input \( y \) terminates with output \( z \).
Note that this relation has a $\Sigma^{0}_1$ definition in Peano arithmetic, $\mathsf{PA}$, as a formula \( \exists w (\pred{T}_1(x,y,w) \land \pred U (w) = z ) \) where \( \pred{T}_1 \) and $\pred U$ are primitive recursive.
It will be notationally convenient to use \( x \cdot y \) in place of a \emph{term} (with the obvious interpretation) though use of this abbreviation will be constrained to contexts in which potential for confusion is minimal.

With the above ternary relation we can express the property of two closed terms having equal value, via a $\Sigma^0_1$-formula $\pred{Eq}(y,z)$.
That is, \( \pred{Eq}(y,z) \) expresses that $y$ and $z$ are codes of closed $\mathcal{L}$-terms $s$ and $t$ respectively such that $s=t$ is a true equation. 

As is usual, the proof-theoretic strength of the theories considered in this paper is expressed by means of recursive ordinals.
However, this paper only mentions two ordinals $\varepsilon_{\varepsilon_0}$ and $\varphi 20$.
The ordinal function \( \alpha \mapsto \varepsilon_\alpha \) is the enumerating function of the \emph{epsilon} numbers, namely the fixed points of \( \omega \)-exponentiation. 
Thus, \( \varepsilon_{\varepsilon_0} \) is the $\varepsilon_0$-th \emph{epsilon} ordinal.
The latter ordinal is the first ordinal $> 0$ closed under the above function. 
That is, for every $\alpha < \varphi 20$ we have $ \alpha < \varepsilon_{\alpha} < \varphi 20$.
For our purposes it suffices to fix an appropriate ordinal notation system $\mathcal{O}$ containing a representation of \( \varphi 2 0 \) (see, e.g.,~\cite{pohlers2008proof}).
Then, $\mathsf{PA}$ can define a formula $x \in \pred{OT} $ as meaning that $x$ represents an ordinal number in $\pred{OT}$.
We extend the numerals to ordinals \( \omega \le \alpha \le \varphi 2 0 \) by setting \( \underline \alpha = \underline n \) where \( n \in \pred{OT} \) represents \( \alpha \).
Symbols $\alpha, \beta,$ and $\gamma$ range over the ordinal numbers in $\pred{OT}$.
Thus, $\forall \alpha A(\alpha)$ abbreviates $\forall x (x \in \pred{OT} \to A(x)).$
By the standard method, $\mathsf{PA}$ can also define standard relations and operations on $\pred{OT}$. 
In particular, let $<$ be the induced well-ordering on $\pred{OT}$.
The $\mathrm{T}$-free consequence of systems in this paper can be characterised by provable transfinite induction:
For a formula $A(x)$ of a language $\mathcal{L}'$, 
we define the $\mathcal{L}'$-formula $\pred{TI}(A, x) = \pred{Prog}_x  A(x) \to A(x)$, 
%\hayashi{Hayashi: I will use roman for some defined predicates and axioms. Sans-serif will be used for theories and special constants in Lemma~\ref{lem:partial_compositionality_GCR}. 
%As for the truth predicate and falsity predicate, I will use italic. If there are any inconveniences, I will change them!} \leigh{This change is good. I've removed `\( \lambda x \)' from \( \pred{Prog} \) to improve readability.}
where $\pred{Prog}_x A(x) = \forall \alpha(\forall \beta < \alpha (A(\beta)) \to A(\alpha))$.
Then, we define the $\mathcal{L}'$-schema:
%\[
$\pred{TI}(x) = \{ \pred{TI}(A, x) \ | \ A\text{ is an $\mathcal{L}'$-formula} \}$.
%\]
For an ordinal \( \alpha \le \varphi 2 0 \), we set
$\pred{TI}({<}\alpha) := \bigcup_{\beta < \alpha} \pred{TI}(\underline\beta).
$
Finally, for an $\mathcal{L}'$-theory $\mathsf{T}$, we define $\mathsf{T} + \mathsf{TI}({<} \alpha)$ to be the extension of $\mathsf{T}$ with the $\mathcal{L}'$-schema $\pred{TI}({<} \alpha)$.

%%%%%%%%%%%%%%%%%%%%%%%%%%%%%%%%%%%%%%%%%%%%%%%%%%%%%%%%%%%%%%

%\subsection{Classical Compositional Truth}
Finally, we introduce the standard Tarskian truth and its axiomatisation.
Truth for $\mathcal{L}$ is characterised inductively in the standard manner:
\begin{itemize}
	\item A closed equation $s=t$ is true iff $s$ and $t$ denote the same value in $\mathbb{N}$;
	\item $A \to B$ is true iff if $A$ is true, then $B$ is true;
	\item $\forall x A(x)$ is true iff $A(s)$ is true for all closed terms $s$.
\end{itemize}

Quantification over $\mathcal{L}$-sentences and the operations on syntax implicit in the above clauses can be expressed via a G\"{o}del-numbering. 
Thus, employing a unary (truth) predicate $\T$, a formal system $\mathsf{CT}$ can be defined corresponding, in a straightforward manner, to the Tarskian truth clauses.
%For simplicity, we sometimes suppress the parentheses of $\mathrm{T}(x)$ and just write $\mathrm{T}x$.

\begin{definition}[$\mathsf{CT}$]\label{defn:CT}
For a unary predicate $\T ,$ let $\Lt = \mathcal{L} \cup \{ \T \}.$
The $\Lt$-theory $\mathsf{CT}$ (compositional truth) consists of $\mathsf{PA}$ formulated for the language $\Lt$ plus the following three axioms.
\begin{description}
%\item[$(\mathsf{CT}_{Eq})$] $\forall \gn {s} \forall \gn t \forall \gn {A_v} \in \pred{Sent}_{\mathcal{L}}.\  \pred{Eq}(\ulcorner s \urcorner, \ulcorner t \urcorner) \to (\mathrm{T} \ulcorner A(s) \urcorner \to \mathrm{T} \ulcorner A(t) \urcorner)$.

\item[$(\mathrm{CT}_{=})$] $\forall \gn s, \gn t. \  \T ( s \deq t )  \leftrightarrow \pred{Eq}(s,t) $.

\item[$(\mathrm{CT}_{\to})$] $\forall \gn {A },\gn {B} \in \pred{Sent}_{\mathcal{L}}. \ \T\gn {A \to B} \leftrightarrow ( \T \gn A \to \T\gn B ) $.

\item[$(\mathrm{CT}_{\forall})$] $\forall \gn {A_v} \in \pred{Sent}_{\mathcal{L}}. \ \T\gn {\forall v A} \leftrightarrow \forall x ( \T \gn {A(\dot{x})} )$.

\end{description}
\end{definition}

As a consequence of the above axioms,
%The other axioms correspond to the above conditions for Tarski's truth.
%
%The first is the observation that 
$\mathsf{CT}$ staisfies Tarski's Convention T \cite[pp.~187--188]{tarski1983concept} for formulas in \( \mathcal L \):

\begin{lemma}\label{fact:tarski_biconditional}
The Tarski-biconditional is derivable in \( \mathsf{CT} \) for every formula of \( \mathcal L \). That is, for each formula \( A(x_1, \dotsc, x_k) \) of \( \mathcal L \) in which only the distinguished variables occur free, we have
\[
\mathsf{CT} \vdash \forall x_1, \dotsc , x_k ( \T \ulcorner A(\dot{ x}_1, \dotsc, \dot{ x}_k ) \urcorner \leftrightarrow A(x_1 , \dotsc, x_k) ).
\]
\end{lemma}

%Second is the `term regularity principle' stating that the truth value of each formula depends only on the value of terms and not their `structure'.

%For a formula \( A \), variable \( x \) and term \( s \).

%Let \( \subt \) be a primitive recursive function function such that \( \subt \colon \gn{A(x)}, \gn x , \gn s \mapsto \gn{A(s)} \) for each formula \( A \), variable \( x \) and term \( s \).

%\begin{lemma}\label{term-regular-CT}
%	Provable in \( \mathsf{CT} \) is the term regularity principle:
%	\[
%	\forall \gn {s} \forall \gn t \forall \gn {A_v} \in \pred{Sent}_{\mathcal{L}}.\  \pred{Eq}(\gn s , \gn t ) \to (\mathrm{T} \gn {A( s)} \to \mathrm{T} \gn {A( t)})
%	\]
%	where the term \( \gn{A(s)} \) is shorthand for \( \subt(\gn{A}, v, \gn s) \), and \( \gn{A(t)} \) likewise.
%\end{lemma}

%%%%%%%%%%%%%%%%%%%%%%%%%%%%%%%%%%%%%%%%%%%

%\section{Classical Realisability}\label{sec:classical_realisability}
%We present a classical number realisability interpretation for $\mathsf{PA}$ and the corresponding axiomatic theory $\mathsf{CR}$.

\subsection{Classical Number Realisability}\label{subsec:classical_number_realisability}

Following \cite{hayashi2024compositional}, we give definitions and results on the system of compositional realisability $\mathsf{CR}$.
Before we get to the axiomatisation, let's look at the set-theoretic definition of classical number relisability and how each theorem of $\mathsf{PA}$ is realisable in it.

\begin{definition}\label{defn:pole}
A \textbf{pole} $\pole$ is a subset of $\mathbb{N}$ closed under converse computation:
for all \( e,m,n \in \mathbb N \), if $ e \cdot m \simeq n$ and $n \in \pole,$ then $\langle e, m \rangle \in \pole$. 
\end{definition}

Note that the empty set $\emptyset$ and natural numbers $\mathbb{N}$ trivially satisfy the above condition; thus, they are poles.

Given a pole $\pole$ and $\mathcal{L}$-sentence $A$, we define sets $\lVert A \rVert_{\pole} , \lvert A \rvert_{\pole} \subseteq \mathbb{N}$ of, respectively, \emph{refutations} (or counter-proofs) and \emph{realisations} (or proofs) of $A$.
The sets are defined such that every pair $ \langle n, m \rangle \in \lvert A \rvert_\pole \times \lVert A \rVert_\pole $ of a realisation and refutation is an element of the pole $\pole$.
Thus, $\pole$ can be seen as the set of contradictions.

\begin{definition}\label{defn:classical_number_realisability}
Fix a pole $\pole.$
For each $\mathcal{L}$-sentence $A$, the sets $\lvert A \rvert_{\pole}, \lVert A \rVert_{\pole} \subseteq \mathbb{N}$ are defined as follows. 
The set $\lvert A \rvert_{\pole}$ is defined directly from $\lVert A \rVert_{\pole}$\textup{:}
\[
\lvert A \lvert_{\pole}\ = \{ n \in \mathbb{N} \mid \forall m \in \lVert A \rVert_{\pole}.\ \langle n, m \rangle \in \pole \}.
\]
The set $\lVert A \rVert_{\pole}$ is defined inductively:
\begin{itemize}
\item 
%\begin{equation}
$\lVert s=t \rVert_{\pole} =
\begin{cases}
\mathbb{N}, & \text{if }\mathbb{N} \not\models s=t,\\
\pole, & \text{otherwise.}
\end{cases}$\smallskip
%\end{equation}

\item $\lVert A \to B \rVert_{\pole} = \{ n \mid (n)_0 \in \lvert A \rvert_{\pole} \text{ and } (n)_1 \in \lVert B \rVert_{\pole} \}$.\medskip

\item $\lVert \forall x A \rVert_{\pole} = \{ n \mid (n)_1 \in \lVert A((n)_0) \rVert_{\pole} \}$.
%\item $||\forall x A|| := \underset{n}{\bigcup} ||A(\bar{n})|| $

\end{itemize}
\end{definition}
The motivation for the definitions of $\lVert A \rVert_{\pole}$ and $\lvert A \rvert_{\pole}$ should be clear.
A false equation is refuted by every number whereas a true equation is refuted only by `contradictions', i.e., elements of the pole.
A refutation of $A \to B$ is a pair \( \langle m , n \rangle \) for which $m$ realises $A$ and $n$ refutes $B$.  
A refutation of $\forall x A$ is a pair \( \langle m , n \rangle \) such that $n$ refutes $A( m)$.
Finally, a realiser of $A$ is a number \( n \) that contradicts all refutations of $A$, i.e., $\langle n, m \rangle \in \pole$ for every $m \in \lVert A \rVert_{\pole}$.
In particular, the closure condition on poles implies that every partial recursive function \( \lVert A \rVert_\pole \to \pole \) is a realiser of \( A \): if for every \( n \in \lVert A \rVert_\pole \), \( e \cdot n \) is defined and an element of \( \pole \) then, by definition, \( \langle e , n \rangle \in \pole \) for every \( n \in  \lVert A \rVert_\pole \).
It is also clear that every realiser of \( A \) induces a canonical partial recursive function mapping \(  \lVert A \rVert_\pole \) to \( \pole \).

\begin{fact}[{\cite[cf~L.~3 \& 4]{hayashi2024compositional}}]\label{fact:number-realisability}
There exist numbers $\mathsf{k}_{\pi}$, $\mathsf{k}_{\pole}$, and $\mathsf{i}$ such that each of
the following holds for any numbers $a,b$, for any $\mathcal{L}$-sentences $A,B$, and for any pole $\pole$:
\begin{enumerate}
\item If $a \in \lVert A \rVert_{\pole}$, then $\mathsf{k}_{\pi} \cdot a \in \lvert A \to B \rvert_{\pole}$;\,\footnote{Note that $\mathsf{k}_{\pi}$ corresponds to the CPS translation of \emph{call with current continuation} (cf. \cite{griffin1989formulae,krivine2003dependent}).}

\item If $a \in \pole$, then $\mathsf{k}_{\pole} \cdot a \in \lvert A \rvert_{\pole}$.

\item If $a \in \lvert A \to B \rvert_{\pole}$ and $b \in \lvert A \rvert_{\pole}$, then $\mathsf{i} \cdot \langle a,b \rangle \in \lvert B \rvert_{\pole}$.
 
\end{enumerate}
\end{fact}

From the above it is a simple exercise to show realizability of Peano arithmetic:

\begin{corollary}\label{cor:real-PA}
Let $A$ be any $\mathcal{L}$-sentence.
If $\mathsf{PA} \vdash A$, then there exists a number $n$ such that $n \in \lvert A \rvert_{\pole}$ for any pole $\pole$.
\end{corollary}

%\begin{proof}
%The proof is by induction on the derivation of $A$.
%We only look at the case of Peirce's law, so we need to find a number $\mathsf{p}$ such that
%$\mathsf{p} \in \lvert ((A \to B) \to A) \to A \rvert_{\pole}$.
%Thus, taking any $a$ such that $a \in \lVert ((A \to B) \to A) \to A \rVert_{\pole}$, we have to assure that $\langle \mathsf{p}, a \rangle \in \pole$.
%We now have $(a)_0 \in \lvert (A \to B) \to A \rvert_{\pole}$ and $(a)_1 \in \lVert A \rVert_{\pole}$.
%Thus, by the item~1 of Fact~\ref{fact:number-realisability}, we get $\mathsf{k}_{\pi} \cdot (a)_1 \in \lvert A \to B \rvert_{\pole}$.
%Therefore, by the item~3 of Fact~\ref{fact:number-realisability}, we also have $\mathsf{i} \cdot \langle (a)_0, \mathsf{k}_{\pi} \cdot (a)_1 \rangle \in \lvert A \rvert_{\pole}$.
%Hence, it follows that $\langle \mathsf{i} \cdot \langle (a)_0, \mathsf{k}_{\pi} \cdot (a)_1 \rangle, (a)_1 \rangle \in \pole$.
%Finally, we put $\mathsf{p} := \lambda a. \langle \mathsf{i} \cdot \langle (a)_0, \mathsf{k}_{\pi} \cdot (a)_1 \rangle, (a)_1 \rangle$, then we obtain $\langle \mathsf{p}, a \rangle \in \pole$ by the closure condition of pole. \qed
%\end{proof}

%%%%%%%%%%%%%%%%%%%%%%%%%%%%%%%%%%%%%%%%%%%%%%%%%%%%%%%%%%%%

%\subsection{Compositional Theory for Realisability}

%Here, we formalise the classical number realisability presented in Section~\ref{subsec:classical_number_realisability}.
%A corollary of Tarski's undefinability of truth, the language of $\mathsf{PA}$ is insufficient to express  classical realisability fully without expanding the non-logical vocabulary.
To formalise the above classical number realisability interpretation, we follow \cite{hayashi2024compositional} in introducing some additional vocaburaries.
Let \( \Lr \) extend $\mathcal{L}$ by three new predicate symbols:
\begin{itemize}
\item a unary predicate $\pole$ for a pole (written \( x \in \pole \));
\item a binary predicate $ \F $ for refutation (written \( x \F y \));
\item a binary predicate $ \T $ for realisation  (written \( x \T y \)).
%\footnote{Although $\mathrm{T}$ is definable by $\mathrm{F}$ and $\pole,$ we introduce $\mathrm{T}$ as a primitive to simplify the notation.}
\end{itemize}
%Unless otherwise stated, we assume that $\mathsf{PA}$ is formulated for $\mathcal{L}^{+}$.

Then, we can straightforwardly formalise classical number realisability as an $\mathcal{L}_R$-theory.

\begin{definition}[Compositional Realisability]\label{defn:CR}
The $\Lr$-theory $\mathsf{CR}$ extends $\mathsf{PA}$ formulated over $\Lr$ by the universal closures of the following axioms:
\begin{description}
\item[$(\mathrm{Ax}_{\pole})$] 
$\forall x,y,z( x \cdot y \simeq z \to ( z \in \pole \to \langle x, y \rangle \in \pole) )$

\item[$(\mathrm{Ax}_{\T})$] $\forall \gn A \in \pred{Sent}_{\mathcal{L}}. \ \forall a( a \T \gn A \leftrightarrow \forall b  ( b \F \gn A \to \langle a, b \rangle \in \pole ) ) $

%\item[$(\mathsf{CR}_{Eq})$] $\forall \gn {s=t}, \gn {A_x} \in \pred{Sent}_{\mathcal{L}}. \ \pred{Eq}(\gn s ,\gn t) \to \forall x( x \mathrm{F} \gn {A(s)} \to x \mathrm{F} \gn {A(t)} ) $

\item[$(\mathrm{CR}_{=})$] $\forall \gn s, \gn t. \ \forall  a ( a \F\gn {s = t}  \leftrightarrow ( \pred{Eq}(\gn s, \gn t) \to a \in \pole ) )$

\item[$(\mathrm{CR}_{\to})$] $\forall \gn A ,\gn B \in \pred{Sent}_{\mathcal{L}}. \ \forall a( a \F\gn {A \to B} \leftrightarrow ((a)_0 \T \gn A \land (a)_1\F\gn B ))$

\item[$(\mathrm{CR}_{\forall})$] $\forall \gn {A_v} \in \pred{Sent}_{\mathcal{L}}. \ \forall a( a \F\gn {\dfa v A} \leftrightarrow (a)_1 \F \gn {A ((\dot{a})_0  )} ) $
%$ \overrightarrow{s} \mathrm{F}\ulcorner \forall x A \urcorner \leftrightarrow \exists n \overrightarrow{s}\mathrm{F} \ulcorner A(\bar{n}) \urcorner$
\end{description}
%$\mathrm{P} \in \{ \mathrm{T}, \mathrm{F} \}$; 
%$A$ and $B$ are codes of $\mathcal{L}$-sentences.
\end{definition}

\begin{remark}\label{rmk:equiv-CR_=}
The universal closure of the axiom $(\mathsf{CR}_{=})$ is equivalent to the conjunction of the following:
\begin{description}
\item[$(\mathrm{CR}_{=1})$] $\forall \gn s, \gn t. \ \lnot \pred{Eq} ( \gn s , \gn t) \to \forall a ( a \F\gn{ s = t } )$; 

\item[$(\mathrm{CR}_{=2})$] $\forall \gn s, \gn t. \ \pred{Eq}( \gn s , \gn t) \to \forall a ( a \F\gn { s = t } \leftrightarrow a \in \pole) $.
\end{description}
\end{remark}

%A straightforward formal induction in \( \mathsf{CR} \) verifies the term regularity principle for refutations (cf.~Lemma~\ref{term-regular-CT}).

\begin{proposition}[{\cite[Proposition~1]{hayashi2024compositional}}]\label{term-regular}
	Refutations are provably invariant under term values:
	\[
	\forall \gn {s}, \gn t. \ \forall \gn {A_v} \in \pred{Sent}_{\mathcal{L}}.\  \pred{Eq}(\gn s , \gn t ) \to \forall x( x \F \gn {A(s)} \to x \F \gn {A( t)}).
	\]
\end{proposition}

We can further extend $\mathsf{CR}$ in the following two ways.

\begin{definition}\label{defn:reflection-rule}
For an $\Lr$-theory $\mathsf S$, let $\mathsf{S}^+$ the extension of $\mathsf{S}$ with the \emph{reflection rule}
\begin{displaymath}
\infer[]{A}{s \T \ulcorner A \urcorner}
\end{displaymath}
for \( A \) an $\mathcal{L}$-sentence and \( s \) a closed term.
\end{definition}

\begin{definition}\label{defn:empty-pole}
Let $\pole = \emptyset$ denote the sentence $\neg \exists x (x \in \pole)$.
Then, for an $\Lr$-theory $\mathsf S$, we define $\mathsf{S}^{\emptyset}$ to be the extension of $\mathsf S$ with the axiom $\pole = \emptyset$.
\end{definition}

\begin{lemma}[{\cite[Lemma~6]{hayashi2024compositional}}]\label{lem:reflection_empty_CR}
%\begin{enumerate}
Over $\mathsf{CR}$, the following are equivalent.
%\begin{enumerate}
%\item 
\begin{enumerate}
\item The reflection schema: \( \exists x (x\T \ulcorner A \urcorner) \to A  \)  for every $\mathcal{L}$-sentence $A$. \label{reflection_Lschema}
%\item 
\item The axiom: \( \pole = \emptyset \). \label{pole_is_empty}
\end{enumerate}
%\end{enumerate}
%\end{enumerate}
\end{lemma}

\begin{corollary}
$\mathsf{CR}^+$ is a subtheory of $\mathsf{CR}^{\emptyset}$.
\end{corollary}

As is remarked in the next subsection, $\mathsf{CR}^+$ and $\mathsf{CR}^{\emptyset}$ are proof-theoretically equivalent, whereas they behave differently from the model-theoretic viewpoint.

\subsection{Basic Results on Compositional Realisability}

We provide a model of $\mathsf{CR}$ based on classical number realisability.
First, the interpretation of the vocabularies of $\mathcal{L}$ is naturally given by the standard model $\mathbb{N}$ of arithmetic.
Second, we fix any pole $\pole \subseteq \mathbb{N}$ for the interpretation of the predicate $x \in \pole$. The sets $\mathbb{T}_{\pole},\mathbb{F}_{\pole} \subseteq \mathbb{N} \times \mathbb{N}$ (for the interpretations of $x \T y$ and $x \F y$, respectively) are defined by the sets $\lvert A \rvert_{\pole}$ and $\lVert A \rVert_{\pole}$ in Definition~\ref{defn:classical_number_realisability}:
\begin{align}
&\mathbb{T}_{\pole} := \{ ( n,m ) \in \mathbb{N}^2 \mid m \ \text{is a code of an } \mathcal{L}\text{-sentence }A \text{ and }n \in \lvert A \rvert_{\pole} \},   \notag \\
&\mathbb{F}_{\pole} := \{ ( n,m ) \in \mathbb{N}^2 \mid m \ \text{is a code of an } \mathcal{L}\text{-sentence }A \text{ and } n \in \lVert A \rVert_{\pole} \}. \notag
\end{align}
Then, while $\mathsf{CR}^{\emptyset}$ is true only when the pole is empty, $\mathsf{CR}^+$ is shown to be sound with respect to any pole.

\begin{proposition}[{\cite[Proposition~5]{hayashi2024compositional}}]\label{prop:soundness_CR+}
Let $A$ be any $\Lr$-sentence and assume that $\mathsf{CR}^+ \vdash A$.
Then, for any pole $\pole$, we have $\langle \mathbb{N}, \pole, \mathbb{T}_{\pole}, \mathbb{F}_{\pole} \rangle \models A$.
\end{proposition}

As for the proof-theoretic properties of $\mathsf{CR}$,
we first remark that realisability of $\mathsf{PA}$ can be formalised.
\begin{theorem}[{\cite[Theorem~1]{hayashi2024compositional}}]\label{formal_realisability_PA}
We assume that $\mathsf{PA}$ is formulated in the language $\mathcal{L}.$
For each $\mathcal{L}$-sentence $A$, if $\mathsf{PA} \vdash A$, then there exists a closed term $s$ such that $\mathsf{CR} \vdash s \T \gn A $.
%Moreover, this claim is formally expressible in $\mathsf{CR}$, i.e., we can find a number $\mathsf{k}_{\mathsf{PA}}$ such that: 
%\begin{displaymath}
%\mathsf{CR} \vdash \forall x \forall \ulcorner A \urcorner \in \pred{Sent}_{\mathcal{L}}. \ \pred{Bew}_{\mathsf{PA}}(x,\ulcorner A \urcorner) \to (\mathsf{k}_{\mathsf{PA}} \cdot x)\mathrm{T} \ulcorner A \urcorner,
%\end{displaymath}
%where $\pred{Bew}_{\mathsf{PA}}(x,y)$ is a canonical provability predicate for $\mathsf{PA}$ expressing means that $x$ is a code of the proof of a sentence $y$. 
\end{theorem}

Although $\mathsf{CR}$ is expressively strong enough to realise $\mathsf{PA}$,
%In the following, we turn to the proof-theoretic strength of $\mathsf{CR}$ and its relationship with $\mathsf{CT}$.
it is shown to be conservative over $\mathsf{PA}$.
\begin{proposition}[{\cite[Proposition~3]{hayashi2024compositional}}]\label{prop:conservativity_CR}
$\mathsf{CR}$ is conservative over $\mathsf{PA}.$
\end{proposition}

%\subsection{Compositional Realisability as Compositional Truth}

In contrast, $\mathsf{CR}^+$ and $\mathsf{CR}^{\emptyset}$ have the same proof-theoretic strength as $\mathsf{CT}$. 
In particular, $\mathsf{CT}$ is \emph{relatively truth definable} in $\mathsf{CR}^{\emptyset}$ in the sense of \cite{fujimoto2010relative}, that is, relatively interpretable in a way that preserves vocabularies other than the truth predicate.

\begin{lemma}[{\cite[Lemma~5]{hayashi2024compositional}}]\label{lem:CR_with_empty}
$\mathsf{CR}^{\emptyset}$ can define the truth prediate of $\mathsf{CT}$
%\footnote{Therefore, $\mathsf{CT}$ is relatively interpretable in $\mathsf{CR}^{\emptyset}$. In particular, this is a relative truth definition in the sense of \cite{fujimoto2010relative}.} 
as, e.g., the predicate $0 \T x$. In other words, $\mathsf{CR}^{\emptyset}$ derives the following:
\begin{description}
%\item[$(\mathsf{CT}_{Eq})'$] $\forall \ulcorner s = t \urcorner, \ulcorner A(0) \urcorner \in \pred{Sent}_{\mathcal{L}}. \ Eq(\ulcorner s \urcorner, \ulcorner t \urcorner) \to (0\mathrm{T} \ulcorner A(s) \urcorner \to 0\mathrm{T} \ulcorner A(t) \urcorner)$;

\item[$(\mathrm{CT}_{=})'$] $\forall \gn s, \gn t. \ 0 \T\gn {s = t}  \leftrightarrow \pred{Eq}(\gn s, \gn t)$
%$\forall x \forall y ( \num 0\mathrm{T} \ulcorner \dot{x} = \dot{y} \urcorner \leftrightarrow x=y )$; 

\item[$(\mathrm{CT}_{\to})'$] $\forall \gn A , \gn B \in \pred{Sent}_{\mathcal{L}}. \  0 \T\gn {A \to B} \leftrightarrow ( 0 \T \gn A \to 0 \T\gn B )$
%\\$\forall \gn A , \gn B \in \pred{Sent}_{\mathcal{L}}. \ 0\mathrm{T}\ulcorner A \to B \urcorner \leftrightarrow (0\mathrm{T} \ulcorner A \urcorner \to 0\mathrm{T}\ulcorner B \urcorner) $;

\item[$(\mathrm{CT}_{\forall})'$] $\forall \gn {A_x} \in \pred{Sent}_{\mathcal{L}}. \ 0\T\ulcorner \forall x A \urcorner \leftrightarrow \forall x ( 0\T \ulcorner A(\dot{x}) \urcorner )$.
\end{description}
Therefore, every $\mathcal{L}$-theorem of $\mathsf{CT}$ is derivable in $\mathsf{CR}^{\emptyset}$.
\end{lemma}

Compositional truth can therefore be understood as the special case ``$\pole = \emptyset$'' in classical realisability.
As the above interpretation does not work for $\mathsf{CR}^+$, the lower bound of $\mathsf{CR}^+$ is obtained in \cite{hayashi2024compositional} by directly deriving $\mathsf{PA} + \mathsf{TI}({<} \varepsilon_{\varepsilon_0})$.
To summarise, we have the following:

\begin{proposition}[{\cite[Theorem~2]{hayashi2024compositional}}]\label{prop:CT_and_CR_with_empty}
%$\mathsf{RCT}$ with $\pole = \emptyset$ has a relative interpretation in $\mathsf{CT}.$
The theories $\mathsf{CR}^{\emptyset}$, $\mathsf{CR}^+$, and $\mathsf{CT}$ have exactly the same $\mathcal{L}$-consequences as $\mathsf{PA} + \mathsf{TI}({<} \varepsilon_{\varepsilon_0})$.
\end{proposition}

\section{Friedman--Sheard-style Realisability}\label{ch:realisability_sec:FSR}

Starting from the typed theory $\mathsf{CR}$ for realisability, it seems natural to extend it to type-free theories.
In particular, the main purpose of this paper is to formulate a Friedman-Sheard-style theory $\mathsf{FSR}$ for classical realisability and study its proof-theoretic properties. In addition, we will show that $\mathsf{FSR}$ has a nice property, \emph{self realisability},
which states that if $A$ is derived, then $A$ is realisable provably in the same system (see Section~\ref{sec:self-realisability}).

In this section, we want to formulate a self-referential system of classical realisability.
For this purpose, let us first refer to the case of truth theory.
The so-called system $\mathsf{FS}$, one of Friedman and Sheard's theories of truth (cf.~\cite{friedman1987axiomatic,halbach1994system}), consists of the compositional axioms, where the range of quantification of sentences is not only $\mathcal{L}$, but also $\Lt .$ In addition, $\mathsf{FS}$ has the introduction and elimination rules for the truth predicate.
Here, we restate definitions and basic results on $\mathsf{FS}$ for the current language: 

\begin{definition}[cf.~{\cite[Definition~2.1]{halbach1994system}}]\label{defn:GCT_FS}
The $\Lt$-theory $\mathsf{GCT}$ (generalised compositional truth) consists of $\mathsf{PA}$ %for $\mathcal{L}_{\mathrm{T}}$ 
and the following axioms:
\begin{description}
%\item[$(\mathsf{GCT}_{Eq})$] $\forall \ulcorner s = t \urcorner, \ulcorner A(0) \urcorner \in Sent_{\mathcal{L}_{\mathrm{T}}} \{Eq( \ulcorner s \urcorner, \ulcorner t \urcorner) \to (\mathrm{T} \ulcorner A(s) \urcorner \to \mathrm{T} \ulcorner A(t) \urcorner)\}$
\item[$(\mathrm{GCT}_{=})$] $\forall \ulcorner s \urcorner, \ulcorner t \urcorner. \ \T\ulcorner s = t \urcorner \leftrightarrow Eq(\ulcorner s \urcorner, \ulcorner t \urcorner)$; 
%\item[$(\mathsf{GCT}_{Val})$] $\forall s, t \in CT  \forall \ulcorner A(0) \urcorner \in Sent_{\mathcal{L}_{\mathrm{T}}} \{\exists x Val(x, s,t) \to (\mathrm{T} \ulcorner A(\dot{s}) \urcorner \to a \mathrm{T} \ulcorner A(\dot{t}) \urcorner)\}$

%\item[$(\mathsf{GCT}_{P})$] $\mathrm{T} \ulcorner P(\dot{\overrightarrow{x}}) \urcorner \leftrightarrow P(\overrightarrow{x})$ for an $\mathcal{L}$-atomic predicate $P;$ 

\item[$(\mathrm{GCT}_{\to})$] $\forall \ulcorner A \urcorner, \ulcorner B \urcorner \in \pred{Sent}_{\Lt}. \ \T\ulcorner A \to B \urcorner \leftrightarrow (\T \ulcorner A \urcorner \to \T\ulcorner B \urcorner) ;$

\item[$(\mathrm{GCT}_{\forall})$] $\forall \ulcorner A_x \urcorner \in \pred{Sent}_{\Lt}. \ \T\ulcorner \forall x A \urcorner \leftrightarrow \forall v (\T \ulcorner A(\dot{v}) \urcorner )$.

\end{description}
The $\Lt$-theory $\mathsf{FS}$ (Friedman-Sheard theory of truth) is the extension of $\mathsf{GCT}$
with the following rules: for each $\Lt$-sentence $A,$
\[
\infer[(\mathrm{NEC})]{\T \ulcorner A \urcorner}{A}
\ \ \ \infer[(\mathrm{CONEC})]{A}{\T \ulcorner A \urcorner}.
\] 
\end{definition}

\begin{lemma}\label{term-regular-GCT}
	Provable in \( \mathsf{GCT} \) is the term regularity principle:
	\[
	\forall \gn {s}, \gn t. \ \forall \gn {A_v} \in \pred{Sent}_{\Lt}.\  \pred{Eq}(\gn s , \gn t ) \to (\T \gn {A( s)} \to \T \gn {A( t)}).
	\]
%	where the term \( \gn{A(s)} \) is shorthand for \( \subt(\gn{A}, v, \gn s) \), and \( \gn{A(t)} \) likewise.
\end{lemma}

\begin{fact}[\cite{Pertun-Broberg:2025On-proofs}]\label{fact:FS_GCT-NEC-TI}
$\mathsf{FS}$ and $\mathsf{GCT} + \mathsf{NEC} + \mathsf{TI}({<} \varphi20)$ have exactly the same theorems.
\end{fact}

Finally, it is worth mentioning the well known fact about $\mathsf{GCT}$.
In contrast to the axiom $(\mathrm{GCT}_{=})$, the Tarski-biconditional for the truth predicate yields a paradox:

\begin{fact}[cf.~\cite{friedman1987axiomatic}]\label{fact:inconsistency_GCT}
\begin{enumerate}
\item $\mathsf{GCT}$ with the following axiom $(\mathrm{GCT}_{\T})$ is inconsistent:
%\begin{description}
%\item[$(\mathsf{GCT}_{\mathrm{T}})$] 
\begin{align}
\forall \ulcorner A \urcorner \in \pred{Sent}_{\Lt}. \ \T \ulcorner \T\ulcorner A \urcorner \urcorner \leftrightarrow \T\ulcorner A \urcorner. \tag{$\mathrm{GCT}_{\T}$}
\end{align}
%\end{description}
\item $\mathsf{GCT}$ with the following schema $(\T \mhyph \mathrm{Out})$ is inconsistent:
\begin{align}
 \T \ulcorner A \urcorner \to A, \text{ for each $\Lt$-sentence $A$.} \tag{$\T \mhyph \mathrm{Out}$}
\end{align}
\end{enumerate}
\end{fact}
The counterpart in classical realisability of the above paradoxical result will be given in Propositions~\ref{inconsistency_GCR1} and~\ref{inconsistency_GCR2}.

\subsection{Generalised Compositional Realisability $\mathsf{GCR}$}

Similar to $\mathsf{GCT}$, we can straightforwardly generalise $\mathsf{CR}$ to $\mathsf{GCR}:$

\begin{definition}[$\mathsf{GCR}$]\label{defn:GCR}
An $\Lr$-theory $\mathsf{GCR}$ (generalised compositional realisability) consists of $\mathsf{PA}$ and the universal closures of the following axioms:
\begin{description}
\item[$(\mathrm{Ax}_{\pole})$] $\pred{T}_1(a,b,c) \to ( \pred U(c) \in \pole \to \langle a, b \rangle \in \pole)$

\item[$(\mathrm{Ax}_{\T})$] $\forall \ulcorner A \urcorner \in Sent_{\Lr}. \ a \T \ulcorner A \urcorner \leftrightarrow \forall b ( b \F \ulcorner A \urcorner \to \langle a, b \rangle \in \pole )$

\item[$(\mathrm{Reg})$] $\forall \ulcorner s \urcorner, \ulcorner t \urcorner. \ \forall \ulcorner A_x \urcorner \in Sent_{\Lr}. \ \pred{Eq}( \ulcorner s \urcorner,\ulcorner t \urcorner) \to ( a \F \ulcorner A(s) \urcorner \to a \F \ulcorner A(t) \urcorner )$

\item[$(\mathrm{GCR}_{P1})$] $\neg P \vec{x} \to a \F\ulcorner P \dot{\vec{x}} \urcorner$

\item[$(\mathrm{GCR}_{P2})$] $P \vec{x} \to (a \F\ulcorner P \dot{\vec{x}} \urcorner \leftrightarrow a \in \pole)$

\item[$(\mathrm{GCR}_{\to})$] $\forall \ulcorner A \urcorner, \ulcorner B \urcorner \in Sent_{\Lr}. \ a \F\ulcorner A \to B \urcorner \leftrightarrow \big( (a)_0 \T \ulcorner A \urcorner \land (a)_1\F\ulcorner B \urcorner \big)$

\item[$(\mathrm{GCR}_{\forall})$] $\forall \ulcorner A_x \urcorner \in Sent_{\Lr} . \ a \F\ulcorner \forall x A \urcorner \leftrightarrow (a)_1 \F \ulcorner A((\dot{a})_0) \urcorner $

\end{description}
In the above, %$\mathrm{P}\in \{ \mathrm{T}, \mathrm{F} \}$; 
$P$ ranges over atomic predicates of $\mathcal{L} \cup \{ x \in \pole \}.$
%; $A$ and $B$ are codes of $\mathcal{L}^{+}$-sentences.
\end{definition}

\begin{remark}
Note that the above axiom $(\mathsf{Ax}_{\pole})$ is exactly the same as in Definition~\ref{defn:CR}, but we have chosen not to use abbreviations because its precise form is important to prove the realisability of $(\mathrm{Ax}_{\pole})$ in Lemma~\ref{realisability_Axpole}.
Similarly, to simplify the proof of Lemmas~\ref{realisability_GCRnegP} and \ref{realisability_GCRP}, we separate (and weaken) the axioms for atomic predicates into $(\mathrm{GCR}_{P1})$ and $(\mathrm{GCR}_{P2})$ (cf. Remark~\ref{rmk:equiv-CR_=}). 
As stated, the axioms $(\mathrm{GCR}_{P1})$ and $(\mathrm{GCR}_{P2})$ do not imply the term regularity principle, contrary to Proposition~\ref{term-regular}.
For example, if $n \notin \pole$ for some numeral $n$, then we have $a F \ulcorner n \in \pole \urcorner$ by $(\mathrm{GCR}_{P1})$.
In this case, however, $(\mathrm{GCR}_{P1})$ says nothing about $a F \ulcorner t \in \pole \urcorner$ if $t$ is not a numeral.
Thus, we explicitly add the principle in the form of axiom $(\mathrm{Reg})$. 
\end{remark}

%As is clear from the above definition, $\mathsf{GCR}$ is just an generalisation of $\mathsf{CR}$ in that each axiom of $\mathsf{GCR}$ ranges over every $\mathcal{L}_R$-sentences. Thus, by the same proof, all the results in $\mathsf{CR},$ except Lemma~\ref{lem:reflection_empty_CR} and Proposition~\ref{prop:CT_and_CR_with_empty}, hold also in 
%$\mathsf{GCR}$ by replacing the range of quantification of sentences from $\mathcal{L}$ to $\mathcal{L}_R.$
%Because of that, even in $\mathsf{GCR},$ we shall sloppily use the result for $\mathsf{CR}.$

The system $\mathsf{GCR}$ does not have compositional axioms for $x \T y$ explicitly,
but we can prove the following useful facts, just as for $\mathsf{CR}$:

\begin{lemma}\label{lem:truth-condition-P}
For each atomic predicate $P \in \mathcal{L} \cup \{ \pole \}$, the following are derived in $\mathsf{GCR}$:
\begin{enumerate}
\item $\neg P \vec{x} \to \big( a \T\ulcorner P \dot{\vec{x}} \urcorner \leftrightarrow \forall y ( \langle a, y \rangle \in \pole ) \big)$

\item $P \vec{x} \to \big( a \T\ulcorner P \dot{\vec{x}} \urcorner \leftrightarrow \forall y ( y \in \pole \to \langle a, y \rangle \in \pole ) \big)$
\end{enumerate}
\end{lemma}

\begin{lemma}
[cf.~{\cite[Lemma~3]{hayashi2024compositional}}]\label{lem:partial_compositionality_GCR}
	There exist numbers $\mathsf{i}$, $\mathsf{h}$, $\mathsf{u}$, $\mathsf{s}$, and $\mathsf{e}$ such that $\mathsf{GCR}$ derives each of the following:
\begin{enumerate}
\item
\(
\forall \gn A, \gn B \in \pred{Sent}_{\Lr}. \ \forall a , b \big( a \T \gn{ A \to B } \wedge b \T \gn A \to ({\mathsf{i}} \cdot \langle a,b\rangle) \T \ulcorner B \urcorner \big).
\)

\item
$\forall \gn A, \gn B \in \pred{Sent}_{\Lr}. \ \forall a \Big( \forall b \big( b \T\ulcorner A \urcorner \to (a \cdot b) \T \ulcorner B \urcorner \big) \to (\mathsf{h} \cdot a) \T \ulcorner A \to B \urcorner \Big).$

\item
\(
\forall \gn {A_x}  \in \pred{Sent}_{\Lr}. \ \forall x \big( ( a\cdot x ) \T \gn {A(\dot{ x})} \big) \to (\mathsf{u}\cdot a)\T \gn {\forall x A}. 
\)

\item
\(
\forall \gn {A_x} \in \pred{Sent}_{\Lr}. \ \forall a \Big( a \T \gn {\forall x A} \to \forall y \big( (\mathsf{s} \cdot \langle a,y \rangle) \T \gn {A(\dot{ y})} \big) \Big).
\)

\item
\(
\forall \gn {A_x}  \in \pred{Sent}_{\Lr}. \ \forall y \big(  a\T \gn {A(\dot{ y})} \to (\mathsf{e}\cdot \langle a,y \rangle)\T \gn {\exists x A} \big). 
\)
\end{enumerate}
\end{lemma}

Using the above lemma, we now generalise the relative truth definability result (Lemma~\ref{lem:CR_with_empty}) to that for $\mathsf{GCR}$.
We define a translation $\mathcal{T}_{\mathsf{FS}}: \Lt \to \Lr$ as follows:
%\leigh{L: I think it should be \( \mathcal{T}_{\mathsf{FS}}(A) = \{ \dotsm \), and `\( A \) \underline{is} \( s = t \) for some \( s \), \( t \)'}
\begin{equation}
\mathcal{T}_{\mathsf{FS}}(A) =
\begin{cases}
s=t & \text{if $A$ is $s=t$ for some $s,t$,} \notag \\
\mathcal{T}_{\mathsf{FS}}(B) \to \mathcal{T}_{\mathsf{FS}}(C) & \text{if $A$ is $ B \to C$} \notag \\
\forall x \mathcal{T}_{\mathsf{FS}}(B(x)) & \text{if $A$ is $ \forall x B(x)$} \notag \\
0 \T (\tau (t)) & \text{if $A$ is $ \T(t)$ for some $t$}, \notag
\end{cases}
\end{equation}
where $\tau(x)$ is a primitive recurive representation of $\mathcal{T}_{\mathsf{FS}}$. Thus, we have $\mathsf{PA} \vdash \tau (\ulcorner A \urcorner) = \ulcorner \mathcal{T}_{\mathsf{FS}}(A) \urcorner$ for each $A \in \mathcal{L}_T$.

\begin{proposition}\label{prop:truth-def-GCT}
Let $\mathcal{T}_{\mathsf{FS}}$ be as above.
For any $\mathcal{L}_T$-formula $A$, if $\mathsf{GCT} \vdash A$, then $\mathsf{GCR}^{\emptyset} \vdash \mathcal{T}_{\mathsf{FS}}(A)$.
\end{proposition}

\begin{proof}
The proof proceeds by induction on the derivation of $A$.
For the base case, it is enough to check the axioms for the truth predicate.
For example, let us consider the axiom ($\mathrm{GCT}_{\to}$): 
\[
\forall \ulcorner A \urcorner, \ulcorner B \urcorner \in Sent_{\Lt}. \ \T\ulcorner A \to B \urcorner \leftrightarrow (\T \ulcorner A \urcorner \to \T \ulcorner B \urcorner).
\]
Thus, taking any codes $\ulcorner A \urcorner, \ulcorner B \urcorner \in Sent_{\Lt}$,
we prove in $\mathsf{GCR}^{\emptyset}$ that: 
\begin{align}
\mathcal{T}_{\mathsf{FS}}(\T\ulcorner A \to B \urcorner) \leftrightarrow \big(\mathcal{T}_{\mathsf{FS}}(\T \ulcorner A \urcorner) \to \mathcal{T}_{\mathsf{FS}}(\T\ulcorner B \urcorner)\big). \label{thm:truth-def-FS:case-to}
\end{align}
Here, $\mathcal{T}_{\mathsf{FS}}(\T\ulcorner A \to B \urcorner)$ is equivalent to $0 \T \ulcorner \mathcal{T}_{\mathsf{FS}} (A) \to \mathcal{T}_{\mathsf{FS}} (B) \urcorner$.
Similarly, we have $\mathcal{T}_{\mathsf{FS}}(\T \ulcorner A \urcorner) \leftrightarrow 0 \T \ulcorner \mathcal{T}_{\mathsf{FS}}(A) \urcorner$ and $\mathcal{T}_{\mathsf{FS}}(\T \ulcorner B \urcorner) \leftrightarrow 0 \T \ulcorner \mathcal{T}_{\mathsf{FS}}(B) \urcorner$.
\begin{itemize}
\item To get the left-to-right direction of (\ref{thm:truth-def-FS:case-to}), we assume $0 \T \ulcorner \mathcal{T}_{\mathsf{FS}} (A) \to \mathcal{T}_{\mathsf{FS}} (B) \urcorner$ and $0 \T \ulcorner \mathcal{T}_{\mathsf{FS}}(A) \urcorner$. 
Then, by Lemma~\ref{lem:partial_compositionality_GCR}, %there exists a term $t$ such that 
we have
$(\mathsf{i} \cdot \langle 0,0 \rangle) \T \ulcorner \mathcal{T}_{\mathsf{FS}}(B) \urcorner$.
But, in the presence of $\pole = \emptyset$, this is equivalent to $0 \T \ulcorner \mathcal{T}_{\mathsf{FS}}(B) \urcorner$.
\item For the converse direction, we suppose $0 \T \ulcorner \mathcal{T}_{\mathsf{FS}}(A) \urcorner \to 0 \T \ulcorner \mathcal{T}_{\mathsf{FS}}(B) \urcorner$.
Then, we have $\forall x \Big( x \T \ulcorner \mathcal{T}_{\mathsf{FS}}(A) \urcorner \to \big( (\lambda a.0) \cdot x \big) \T \ulcorner \mathcal{T}_{\mathsf{FS}}(B) \urcorner \Big)$, and hence it follows by Lemma~\ref{lem:partial_compositionality_GCR} that $\big( \mathsf{h} \cdot (\lambda a.0) \big) \T \ulcorner \mathcal{T}_{\mathsf{FS}} (A) \to \mathcal{T}_{\mathsf{FS}} (B) \urcorner$, which implies by $\pole = \emptyset$ that $0 \T \ulcorner \mathcal{T}_{\mathsf{FS}} (A) \to \mathcal{T}_{\mathsf{FS}} (B) \urcorner$.
\end{itemize}
In summary, (\ref{thm:truth-def-FS:case-to}) is obtained.
The other cases and the inductive step are similarly proved, using Lemma~\ref{lem:partial_compositionality_GCR}. \qed
\end{proof}

Before defining the Friedman--Sheard system, let us observe inconsistency results associated with $\mathsf{GCR}^{\emptyset}$.
As in the case of $\mathsf{CR}$, the reflection schema is equivalent to the assumption that the pole is empty:

\begin{lemma}[cf.~{\cite[Lemma~6]{hayashi2024compositional}}]\label{lem:reflection_empty_GCR}
%\begin{enumerate}
Over $\mathsf{GCR}$, the following are equivalent.
%\begin{enumerate}
%\item 
\begin{enumerate}
\item The reflection schema: \( \exists x (x\mathrm{T} \ulcorner A \urcorner) \to A  \)  for every $\mathcal{L}$-sentence $A$. %\label{reflection_Lschema}
%\item 
\item The axiom: \( \pole = \emptyset \). %\label{pole_is_empty}
\end{enumerate}
%\end{enumerate}
%\end{enumerate}
\end{lemma}

However, similar to the case of truth theory (Fact~\ref{fact:inconsistency_GCT}), the next proposition shows that we need to treat $\T$ (and $\F$) differently from the atomic predicates of $\mathcal{L} \cup \{ x \in \pole \}$. 
That is the reason why Lemma~\ref{lem:truth-condition-P} and Lemma~\ref{lem:reflection_empty_GCR} cannot be generalised to $\Lr$-sentences without inconsistency.

\begin{proposition}\label{inconsistency_GCR1}
Let $\mathsf{GCR}'$ be the extension of $\mathsf{GCR}$ with the following axioms:
\begin{description}
\item[$(\mathrm{GCR}_{\T1})$] $\forall \ulcorner A \urcorner \in \pred{Sent}_{\Lr}. \ \neg x \T\ulcorner A \urcorner \to \big( a \T\ulcorner \dot{x} \T(\ulcorner A \urcorner) \urcorner \leftrightarrow \forall y ( \langle a,y \rangle \in \pole ) \big)$,

\item[$(\mathrm{GCR}_{\T 2})$] $\forall \ulcorner A \urcorner \negmedspace \in \negmedspace \pred{Sent}_{\Lr}. \ x\T \ulcorner A \urcorner \negthinspace \to\negthinspace \bigl( a \T\ulcorner \dot{x}\T(\ulcorner A \urcorner) \urcorner \negthinspace \leftrightarrow \negthinspace \forall y ( y \negmedspace \in\negmedspace \pole \to \langle a,y \rangle \negmedspace\in\negmedspace \pole ) \big) .$
\end{description}
Then, $\mathsf{GCR}' \vdash \pole \neq \emptyset.$

\end{proposition}

\begin{proof}
For a contradiction, we suppose 
$\pole = \emptyset$ in $\mathsf{GCR}'$.
Then, from the proof of Theorem~\ref{thm:truth-def-FS}, we can define the truth predicate of $\mathsf{GCT}$ as $\mathcal{T}_{\mathsf{FS}}(\T(x))$ within $\mathsf{GCR}' + \pole = \emptyset$.
In addition, from the axioms $(\mathrm{GCR}_{\T1})$ and $(\mathrm{GCR}_{\T2})$,
we can derive the following in $\mathsf{GCR}' + \pole = \emptyset$: 
\[ \forall \ulcorner A \urcorner \in Sent_{\Lt}. \ 0\T \ulcorner 0\T\ulcorner \mathcal{T}_{\mathsf{FS}}(A) \urcorner \urcorner \leftrightarrow 0\T\ulcorner \mathcal{T}_{\mathsf{FS}}(A) \urcorner,
\]
which is equivalent to $\mathcal{T}_{\mathsf{FS}}(\mathrm{GCT}_{\T})$.
This means that $\mathsf{GCR}' + \pole = \emptyset$ defines the truth predicate of $\mathsf{GCT} + (\mathrm{GCT}_{\T})$.
Thus, by Fact~\ref{fact:inconsistency_GCT}, we obtain a contradiction, and hence $\mathsf{GCT}' \vdash \pole \neq \emptyset$. \qed
\end{proof}

\begin{proposition}\label{inconsistency_GCR2}
The extension of $\mathsf{GCR}$ with the unrestricted reflection schema is inconsistent:
\[
\exists x ( x\T \ulcorner A \urcorner) \to A, \text{ for each $\Lr$-sentence $A$.}
\]
\end{proposition}
\begin{proof}
	By Lemma~\ref{lem:reflection_empty_GCR}, the reflection schema implies in $\mathsf{GCR}$ that $\pole = \emptyset.$
Thus, in the same way as Proposition~\ref{inconsistency_GCR1}, the truth predicate of $\mathsf{GCT}$ is definable in $\mathsf{GCR} + \pole= \emptyset$. Furthermore, the reflection schema implies the schema:
\[ 0\T \ulcorner \mathcal{T}_{\mathsf{FS}}(A) \urcorner \to \mathcal{T}_{\mathsf{FS}}(A), \text{ for each $\Lt$-sentence $A$,}
\]
which is equivalent to:
\[ \mathcal{T}_{\mathsf{FS}}(\T \ulcorner A \urcorner \to A), \text{ for each $\Lt$-sentence $A$.}
\]
This means that $\mathsf{GCR}$ with the unrestricted reflection schema defines the truth predicate of $\mathsf{GCT} + (\T \mhyph \mathrm{Out})$.
Therefore, by Fact~\ref{fact:inconsistency_GCT}, we obtain a contradiction. \qed
\end{proof}

% ------------------
\subsection{Friedman--Sheard Realisability $\mathsf{FSR}$}

We now formulate counterparts of the rules $(\mathrm{NEC})$ and $(\mathrm{CONEC})$. The rule $(\mathrm{NEC})$ informally says that from the \emph{explicit} truth of $A$, i.e., $A$ itself, we can infer the \emph{formal} truth of $A$, i.e., $\T \ulcorner A \urcorner$. Based on this idea, we firstly define realisability in the form of an explicit $\Lr$-formula $n \in \lvert A \rvert.$
Then, the realisability version of $(\mathrm{NEC})$ should allow an inference from the \emph{explicit} realisability $n \in \lvert A \rvert$ to the \emph{formal} realisability $n\T \ulcorner A \urcorner.$ The rule $(\mathrm{CONEC})$ is similarly treated.

\begin{definition}[Explicit realisation]
For each term $s$ and  
$\Lr \mhyph$formula $A$,
we inductively define $\Lr$-formulas $s \in \lvert A \rvert$ and $s \in \lVert A \rVert,$
with renaming bound variables in $A$ if necessary.
%\leigh{L: Again, \( = \) instead of \( \leftrightarrow \)}
\begin{itemize}
\item $s \in \lVert P \vec x \rVert  \ = \ P \vec x \to s \in \pole,$ if $P \in \mathcal{L} \cup \{ \pole \}$;
\item $s \in \lVert t \F u \rVert  \ = \  %Sent_{\mathcal{L}^{+}}(u)\land 
s \F(t\dot{\in} \lVert u \rVert); $
\item $s \in \lVert t \T u \rVert  \ = \  %Sent_{\mathcal{L}^{+}}(u) \land 
s \F(t\dot{\in} \lvert u \rvert); $
\item $s \in \lVert A \to B \rVert  \ = \  (s)_{0} \in \lvert A \rvert \land (s)_{1} \in \lVert B \rVert;$
\item $s \in \lVert \forall x A(x) \rVert  \ = \  (s)_1 \in \lVert A((s)_0) \rVert;$
\item $s \in \lvert A \rvert  \ = \  \forall a(a \in \lVert A \rVert \to \langle s, a \rangle \in \pole)$ for a fresh variable $a$.
\end{itemize}
Here, $x \dot{\in} \lvert y \rvert$ is a binary primitive recursive function symbol such that $\mathsf{PA}$ derives $x \dot{\in} |\ulcorner A \urcorner| = \ulcorner \dot{x} \in A \urcorner$ for any $\Lr$-sentence $A$.
The existence of such a function is ensured by the primitive recursion theorem. %(for a proof, see, e.g., \cite{}). 
The other binary primitive recursive function symbol $x \dot{\in} \lVert y \rVert$ is similarly defined.
\end{definition}

%Similarly to Kleene's formal realisability in Heyting arithmetic \cite{kleene1945interpretation}, we obtain the following result by induction.
We can easily verify the next substitution lemma:
\begin{lemma}\label{substitution_explicit_realisation} 
\begin{enumerate}
	\item For every terms $s,t$ and $\Lr$-formula $A,$ we have:
		\begin{enumerate}	
		%\[\mathsf{PA} \vdash x=y \land s=t \to (x \in |A(s)| \to y \in |A(t)|).\]
		\item $\mathsf{PA} \vdash (t\in \lVert A \rVert)[x/s] \ \leftrightarrow \ t[x/s] \in \lVert A[x/s] \rVert$,
		\item $\mathsf{PA} \vdash (t\in \lvert A \rvert)[x/s] \ \leftrightarrow \ t[x/s] \in \lvert A[x/s]) \rvert.$
		\end{enumerate}
		Here, $B[x/s]$ is a substitution of $s$ into the free occurrences of $x$ in  $B$.
		
	\item	For every terms $s,t$ and $\mathcal{L}_R$-formula $A,$ we have:
		\begin{enumerate}	
		\item $\mathsf{PA} \vdash s=t \to (x \in \lVert A(s) \rVert \to x \in \lVert A(t) \rVert).$
		\item $\mathsf{PA} \vdash s=t \to (x \in \lvert A(s) \rvert \to x \in \lvert A(t) \rvert).$
		\end{enumerate}
\end{enumerate}
\end{lemma}		
		
The next lemma shows that explicit realisability and formal realisability are equivalent for $\mathcal{L}$.
\begin{lemma}\label{explicit_formal_realisability}
\begin{enumerate}
	\item Let $P \vec{y} $ be an atomic formula of $\mathcal{L} \cup \{ \pole \}.$
		\begin{enumerate}
			\item $\mathsf{GCR} \vdash x \F \ulcorner P \dot{\vec{y}} \urcorner \leftrightarrow x \in \lVert			
			P\vec{y} \rVert.$
			\item $\mathsf{GCR} \vdash x \T \ulcorner P \dot{\vec{y}} \urcorner \leftrightarrow x \in \lvert 
			P \vec{y} \rvert.$
		\end{enumerate}
		
	\item For every terms $s,t$, we have:
		\begin{enumerate}
			\item $\mathsf{GCR} \vdash x \in \lvert s \F t \rvert \leftrightarrow x \T (s \dot{\in} \lVert t \rVert),$
			\item $\mathsf{GCR} \vdash x \in \lvert s \T t \rvert \leftrightarrow x \T (s \dot{\in} \lvert t \rvert).$
		\end{enumerate}
		
\end{enumerate}
\end{lemma}

\begin{proof}
\begin{enumerate}

	\item As the proof for the second item is similar, we consider only the first.
	So, taking any $x$ and $y$, we derive the claim in $\mathsf{GCR}$ as follows:
	\begin{align}
	x \F \ulcorner P \dot{\vec{y}} \urcorner &\Leftrightarrow P \vec{y} \to x \in \pole 	
	\tag{$\because$ $(\mathrm{GCR}_{P1})$ and $(\mathrm{GCR}_{P2})$} \\
	&\Leftrightarrow x \in \lVert P \vec{y} \rVert. \tag{$\because$ def. of $\lVert P \rVert$}
	\end{align}

	\item We consider the first item. We derive the claim in $\mathsf{GCR}$ as follows:
	\begin{align}
	x \in \lvert s \F t \rvert &\Leftrightarrow \forall a (a \in \lVert s \F t \rVert \to \langle x,a \rangle \in \pole)\tag{$\because$ def. of $\lvert \F \rvert$} \\
	&\Leftrightarrow \forall a (a \F(s \dot{\in} \lVert t \rVert) \to \langle x,a \rangle \in \pole)\tag{$\because$ def. of $\lVert \F \rVert$} \\
	&\Leftrightarrow x \T (s \dot{\in} \lVert t \rVert). \tag{$\because$ $(\mathrm{Ax}_{\T})$}
	\end{align} \qed
\end{enumerate}	 
\end{proof}

\begin{definition}[$\mathsf{FSR}$]\label{FSR}
An $\mathcal{L}_R$-theory $\mathsf{FSR}$ (Friedman--Sheard theory for realisability) consists of $\mathsf{GCR}$ with the two rules
\begin{center} \
\infer[(\mathrm{NECR})]{ s \T \ulcorner A \urcorner}{s \in \lvert A \rvert}
\ \ \ \infer[(\mathrm{CONECR})]{ s \in \lvert A \rvert}{s \T \ulcorner A \urcorner},
\end{center}
for every closed term $s$ and every $\Lr$-sentence $A.$
\end{definition}

%\begin{definition}[$\mathsf{FSR}^{+}$]
%An $\mathcal{L}_R$-theory $\mathsf{FSR}^{+}$ is the extension of $\mathsf{FSR}$ with the reflection rule:
%\begin{displaymath}
%\infer[]{A}{ s \mathrm{T} \ulcorner A \urcorner},
%\end{displaymath}
%for every closed term $s$ and every $\mathcal{L}$-sentence $A.$
%\end{definition}

As with $\mathsf{CR}$, we can equip $\mathsf{FSR}$ with the reflection rule ($= \mathsf{FSR}^+$) or the emptiness assumption for the pole ($= \mathsf{FSR}^{\emptyset}$). 
Their proof-theoretic properties are studied below.

\subsection{Revision Model for $\mathsf{FSR}$}\label{subsec:model_FSR}
Before the proof-theoretic consideration of $\mathsf{FSR},$ we give a \emph{revision} model of $\mathsf{FSR}^{(+)}$,
%Similar to Proposition~\ref{ch:formal_soundness_FS}, we construct a revisional model, 
in which the evaluation of each sentence changes at each step.

Fix any pole $\pole \subseteq \mathbb{N}.$ 
For any given set $\mathbb{F} \subseteq \mathbb{N} \times \mathbb{N}$ for the interpretation of $\F$,
the interpretation $\mathbb{T} \subseteq \mathbb{N} \times \mathbb{N}$ of $\T$ can be defined from $\pole$ and $\mathbb{F}$:
\[
\mathbb{T} := \{ (n, \ulcorner A \urcorner) \in \mathbb{N} \times Sent_{\Lr} \ | \ \forall m \in \mathbb{N} ((m, \ulcorner A \urcorner) \in \mathbb{F} \Rightarrow \langle n,m \rangle \in \pole) \}.
\]
Thus, for a pole $\pole$ and a set $\mathbb{F}$, we shall express the above $\Lr$-model simply by $\langle \pole, \mathbb{F} \rangle.$

The revision operator $\Gamma_{\pole}: \mathbb{N}^2 \to \mathbb{N}^2$ from which the interpretation of \( \F \) is obtain is given by:
\[
\Gamma_{\pole}(\mathbb{F}):= \{ (n, \ulcorner A \urcorner) \in \mathbb{N} \times Sent_{\Lr} \ | \ \langle \pole, \mathbb{F} \rangle \models n \in \lVert A \rVert \}. 
\]
As with the revision hierarchy, we inductively define a series of interpretations:
\begin{align}
\Gamma_{\pole}^{0}(\mathbb{F}) &:= \mathbb{F}, \notag \\
\Gamma_{\pole}^{n+1}(\mathbb{F}) &:= \Gamma_{\pole}(\Gamma_{\pole}^{n}(\mathbb{F})). \notag
\end{align}

\begin{proposition}\label{soundness_FSR}
%Take any $\mathbb{F}_{\pole} \subseteq \mathbb{N} \times \mathbb{N}$
Assume $\mathsf{FSR} \vdash A$ for an $\mathcal{L}_R$-formula.
Then, for any pole $\pole \subseteq \mathbb{N}$ and any set $\mathbb{F} \subseteq \mathbb{N} \times \mathbb{N}$, there exists a number $n$ such that for any $m\geq n$, it holds that $\langle \pole, \Gamma_{\pole}^{m}(\mathbb{F})\rangle \models A^{\forall}$ where $A^{\forall}$ is a universal closure of $A$.
\end{proposition}

\begin{proof}
The proof is by induction on the length of the derivation, so assume $\mathsf{FSR} \vdash A$ and take any pole $\pole \subseteq \mathbb{N}$ and any set $\mathbb{F} \subseteq \mathbb{N} \times \mathbb{N}$.
We divide the cases by the last axiom or rule of the derivation.
\begin{description}
\item[$(\mathrm{Ax}_{\T})$] Suppose that $A$ is the axiom $(\mathrm{Ax}_{\T})$: 
\[
\forall \ulcorner B \urcorner \in Sent_{\Lr}. \ a \T \ulcorner B \urcorner \leftrightarrow \forall b  (b \F \ulcorner B \urcorner \to \langle a, b \rangle \in \pole) ,
\]
then for all $m \geq 1,$ we can show $\langle \Gamma_{\pole}^{m}(\mathbb{F}) \rangle \models A^{\forall}$.
In fact, taking any $m \geq 1,$ $\ulcorner B \urcorner \in Sent_{\Lr},$ and $a \in \mathbb{N},$ we have:
\begin{align}
\langle \pole, \Gamma_{\pole}^{m}(\mathbb{F}) \rangle \models a \T \ulcorner B \urcorner &\Leftrightarrow
\forall b \in \mathbb{N} \big((b, \ulcorner B \urcorner) \in \Gamma_{\pole}^{m}(\mathbb{F}) \Rightarrow \langle a,b \rangle \in \pole \big) \notag \\
%&\Leftrightarrow \langle \Gamma^{m-1}(\mathbb{F}) \rangle \models a \in |B| \notag
&\Leftrightarrow
\langle \pole, \Gamma_{\pole}^{m}(\mathbb{F}) \rangle \models \forall b  (b \F \ulcorner B \urcorner \to \langle a, b \rangle \in \pole). \notag
%&\Leftrightarrow \forall b \in \mathbb{N}  ((b, \ulcorner B \urcorner) \in \Gamma^{m+1}(\mathbb{F}) \Rightarrow \langle a, b \rangle \in \pole). \notag
\end{align}
Thus, it follows that $\langle \pole, \Gamma_{\pole}^{m}(\mathbb{F}) \rangle \models (\mathrm{Ax}_{\T})^{\forall}$.

\item[$(\mathrm{NECR})$] Suppose that $A$ is of the form $ s\T \ulcorner B \urcorner$ and is obtained from $s \in \lvert B \rvert.$
Then, by induction hypothesis we have a number $n$ such that for all $m \geq n,$ 
\begin{align}
& \langle \pole, \Gamma_{\pole}^{m}(\mathbb{F}) \rangle \models s \in \lvert B \rvert \notag \\
&\Leftrightarrow \langle \pole, \Gamma_{\pole}^{m}(\mathbb{F}) \rangle \models \forall b  (b \in \lVert B \rVert \to \langle s, b \rangle \in \pole) \notag \\
&\Leftrightarrow \langle \pole, \Gamma_{\pole}^{m+1}(\mathbb{F}) \rangle \models \forall b  (b \F \ulcorner B \urcorner \to \langle s, b \rangle \in \pole). \notag
\end{align}
Since $\langle \pole, \Gamma_{\pole}^{m+1}(\mathbb{F}) \rangle \models (\mathrm{Ax}_{\T})^{\forall}$, we also get $\langle \pole, \Gamma_{\pole}^{m+1}(\mathbb{F}) \rangle \models s\T \ulcorner B \urcorner,$ which is true for all $m \geq n.$
\end{description}
The other cases are similarly proved.
Note that the axioms of $\mathsf{GCR}$ are satisfied in $\langle \pole, \Gamma_{\pole}^{m}(\mathbb{F})\rangle$ for all $m \geq 1$. \qed
\end{proof}

Recall that the sentence $A$ in the reflection rule of $\mathsf{FSR}^+$ is restricted to $\mathcal{L}$.
Thanks to this condition, we can extend Proposition~\ref{soundness_FSR}) to that for $\mathsf{FSR}^+$:

\begin{proposition}\label{soundness_FSR+}
Proposition~\ref{soundness_FSR}  holds for $\mathsf{FSR}^+$. 
%Thus, $\mathsf{FSR}^+$ is consistent.
\end{proposition}

\begin{proof}
The proof is again by induction on the derivation.
Since the other cases are already treated in the proof of
Proposition~\ref{soundness_FSR}, we consider the case of the reflection rule.
Thus, we assume that an $\mathcal{L}$-sentence $A$ is derived from $s \T \ulcorner A \urcorner$.
By the induction hypothesis, for any pole $\pole \subseteq \mathbb{N}$ and any set $\mathbb{F} \subseteq \mathbb{N} \times \mathbb{N}$, there exists a number $n$ such that for all $m\geq n,$ it holds that $\langle \pole, \Gamma_{\pole}^{m}(\mathbb{F})\rangle \models s \T \ulcorner A \urcorner$. 
In particular, for the empty pole $\pole := \emptyset$ and for all $m \geq 1$, the model $\langle \emptyset,  \Gamma_{\emptyset}^{m}(\mathbb{F})\rangle$ satisfies each axiom of $\mathsf{GCR}$ and the axiom $\pole = \emptyset$.
Thus, by Lemma~\ref{lem:reflection_empty_GCR}, the reflection schema for $\mathcal{L}$-sentences is also true in the same model.
Therefore, the induction hypothesis $\langle \emptyset, \Gamma_{\emptyset}^{m}(\mathbb{F})\rangle \models s \T \ulcorner A \urcorner$ for $m \geq \max (n,1)$ implies that $\langle \emptyset, \Gamma_{\emptyset}^{m}(\mathbb{F})\rangle \models A$.
Since $A \in \mathcal{L}$, this means that $A$ is true in $\langle \pole, \Gamma_{\pole}^{m}(\mathbb{F})\rangle$  for any $\pole$, $\mathbb{F}$, and $m \geq 0$, as desired. \qed

\begin{remark}
Proposition~\ref{soundness_FSR} does not hold if the reflection rule of $\mathsf{FSR}^+$ is further extended to $\mathcal{L} \cup \{ \pole \}$-sentences.
In fact, since $\pole = \emptyset$ is realisable in $\mathsf{FSR}$ (see Corollary~\ref{cor:real_pole-empty} below), this extended reflection rule infers $\pole = \emptyset$, which is, of course, satisfied only when the pole is empty. 
\end{remark}
 
\end{proof}

%%%%%%%%%%%%%%%%%%%%%%%%%%%%%%%%%%%%%%%%%

\subsection{Proof-theoretic Interpretation of $\mathsf{FSR}$}\label{subsec:PTS-of-FSR}
In this subsection, we prove that $\mathsf{FSR}$ and  $\mathsf{FSR}^{\emptyset} + \mathsf{TI}({<} \varphi 20)$ are relatively interpretable in $\mathsf{PA}$ and $\mathsf{FS}$, respectively.

\begin{proposition}\label{conservativity_FSR}
$\mathsf{FSR}$ is conservative over $\mathsf{PA}.$
\end{proposition}

\begin{proof}
The proof is exactly the same as for Proposition~\ref{prop:conservativity_CR} (see \cite[Proposition~3]{hayashi2024compositional}).
We define a translation $\mathcal{T}_{\mathbb{N}} \colon \Lr \to \mathcal{L}$ such that the vocabularies of $\mathcal{L}$ are unchanged, and then we show the following:
\begin{center} 
for each $\Lr$-formula $A$, if $\mathsf{FSR} \vdash A,$ then $\mathsf{PA} \vdash \mathcal{T}_{\mathbb{N}}(A).$
\end{center}
If $A \in \mathcal{L},$ we have $\mathcal{T}_{\mathbb{N}}(A) = A$; thus, the conservativity follows.

The required translation $\mathcal{T}_{\mathbb{N}}$ is defined as follows:
\begin{itemize}
\item $\mathcal{T}_{\mathbb{N}}(s=t) \ = \ (s=t)$;
\item $\mathcal{T}_{\mathbb{N}}(s \in \pole) \ = \ \mathcal{T}_{\mathbb{N}}(s\F t) \ = \ \mathcal{T}_{\mathbb{N}}(s\T t) \ = \ (0=0);$
\item $\mathcal{T}_{\mathbb{N}}$ commutes with the logical symbols.
\end{itemize}
Roughly speaking, every pair is contradictory, and every sentence is realised and refuted by every number under this interpretation.
Note that we trivially have $\mathsf{PA} \vdash \mathcal{T}_{\mathbb{N}}(s \in \lvert A \rvert)$ for every term $s$ and formula $A$.
Consequently, we can easily verify that the translation of each axiom of $\mathsf{FSR}$ is derivable in $\mathsf{PA}.$ \qed
\end{proof}

Clearly, the translation $\mathcal{T}_{\mathbb{N}}$ does not work for $\mathsf{FSR}^{\emptyset}$.
In fact, as will be shown below, $\mathsf{FSR}^{\emptyset}$ is as strong as $\mathsf{FS}$.
To determine the proof-theoretic upper bound of $\mathsf{FSR}^{\emptyset}$, we define another translation $\mathcal{T}_{\emptyset} : \Lr \to \Lt$ as follows:
\begin{itemize}
\item $\mathcal{T}_{\emptyset}(s \in \pole) \ = \ \bot$;

\item $\mathcal{T}_{\emptyset}(s \F t) \ = \ \T (\tau (s \dot{\in} \lVert t \rVert))$;

\item $\mathcal{T}_{\emptyset}(s \T t) \ = \ \T (\tau (s \dot{\in} \lvert t \rvert))$;

\item $\mathcal{T}_{\emptyset}$ commutes with the logical symbols,
\end{itemize}
where $\tau$ is a primitive recursive representation of $\mathcal{T}_{\emptyset}$.
Thus, we have $\mathsf{PA} \vdash \tau (\ulcorner A \urcorner) = \ulcorner \mathcal{T}_{\emptyset}(A) \urcorner$ for each $A \in \Lr$.

The following is immediate by the definition of $\mathcal{T}_{\emptyset}$.
\begin{corollary}\label{interpretation_FSR-empty_FS}
$\mathsf{PA}$ formulated over $\Lt$ derives the following for each $\Lr$-formulae $A,B$:
\begin{enumerate}
\item $\forall a. \ \neg \mathcal{T}_{\emptyset}(a \in \pole)$.

\item $\mathcal{T}_{\emptyset}(a \in \lvert A \rvert) \leftrightarrow \forall b \big( \neg \mathcal{T}_{\emptyset}(b \in \lVert A \rVert) \big)$

\item $a=b \to \big( \mathcal{T}_{\emptyset}(A(a)) \to \mathcal{T}_{\emptyset}(A(b)) \big)$

\item $\mathcal{T}_{\emptyset}(a \in \lVert P \vec x \rVert) \leftrightarrow \neg P \vec x,$ for $P \in \mathcal{L}$

\item $\mathcal{T}_{\emptyset}(a \in \lVert A \to B \rVert) \leftrightarrow \big( \mathcal{T}_{\emptyset}((a)_0 \in \lvert A \rvert) \land \mathcal{T}_{\emptyset}((a)_1 \in \lVert B \rVert) \big)$

\item $\mathcal{T}_{\emptyset}(a \in \lVert \forall x A (x) \rVert) \leftrightarrow \mathcal{T}_{\emptyset}((a)_1 \in \lVert A((a)_0)\rVert)$
\end{enumerate}
\end{corollary}

\begin{proposition}\label{upper-bound_FSR-empty}
%$\mathsf{FSR}^{\emptyset}$ is relatively interpretable in $\mathsf{FS}$ by $\mathcal{T}_{\emptyset}$.
For each $\Lr$-formula $A$, if $\mathsf{FSR}^{\emptyset} + \mathsf{TI}({<} \varphi 20) \vdash A$, then $\mathsf{FS} \vdash \mathcal{T}_{\emptyset}(A)$.
\end{proposition}

\begin{proof}
The proof is by induction on the derivation of $A$.
\begin{description}
\item[($\mathrm{Ax}_{\T}$):] We assume that $A$ is $\mathrm{Ax}_{\T}$:
\[
\forall \ulcorner B \urcorner \in Sent_{\Lr}. \ a \T \ulcorner B \urcorner \leftrightarrow \forall b ( b \F \ulcorner B \urcorner \to \langle a, b \rangle \in \pole ). 
\]
By Corollary~\ref{interpretation_FSR-empty_FS}(2), $\mathsf{FS}$ formalises the result:
\begin{align}
\forall \ulcorner B \urcorner \in Sent_{\Lr}. \ \mathrm{T}\ulcorner \mathcal{T}_{\emptyset}(a \in \lvert B \rvert) \urcorner \leftrightarrow \forall b \big( \neg \T \ulcorner \mathcal{T}_{\emptyset}(b \in \lVert B \rVert) \urcorner \big). \label{upper-bound_FSR-empty:case-AxT}
\end{align}
In addition, we have by the definition of $\mathcal{T}_{\emptyset}$ that:
\begin{itemize}
\item $\mathsf{PA} \vdash \forall \ulcorner B \urcorner \in Sent_{\Lr}. \ \T \ulcorner \mathcal{T}_{\emptyset}(a \in \lvert B \rvert) \urcorner \leftrightarrow \mathcal{T}_{\emptyset}(a \T \ulcorner B \urcorner)$, and 

\item $\mathsf{PA} \vdash \forall \ulcorner B \urcorner \in Sent_{\Lr}. \ \T \ulcorner \mathcal{T}_{\emptyset}(b \in \lVert B \rVert) \urcorner \leftrightarrow \mathcal{T}_{\emptyset}(b \F \ulcorner B \urcorner)$.
\end{itemize}
Thus, (\ref{upper-bound_FSR-empty:case-AxT}) is equivalent to $\mathcal{T}_{\emptyset}(\mathsf{Ax}_{\mathrm{T}})$, as required.

\item[$\mathrm{TI}({<} \varphi20)$:] Assume that $A$ is an instance of $\mathrm{TI}({<} \varphi20)$.
Then, $\mathcal{T}_{\emptyset}(A)$ is again an instance of $\mathrm{TI}({<} \varphi20)$ in the language $\mathcal{L}_T$.
Moreover, $\mathsf{FS}$ derives all the instances of $\mathrm{TI}({<} \varphi20)$ (\cite[Theorem~2.41]{leigh2010ordinal}).
Thus, we get $\mathsf{FS} \vdash \mathcal{T}_{\emptyset}(A)$.

\item[$\mathsf{NECR}$:] Assume that $A$ is $t \T \ulcorner B \urcorner$ and is derived from $t \in \lvert B \rvert$.
By the induction hypothesis, we have $\mathsf{FS} \vdash \mathcal{T}_{\emptyset}(t \in \lvert B \rvert)$.
Thus, by $\mathrm{NEC}$, $\mathsf{FS}$ may deduce $\T \ulcorner \mathcal{T}_{\emptyset}(t \in \lvert B \rvert) \urcorner$, which is equivalent to $\mathcal{T}_{\emptyset}(t \T \ulcorner B \urcorner)$.
\end{description}
The other cases are similarly treated with the help of Corollary~\ref{interpretation_FSR-empty_FS}. \qed
\end{proof}

Conversely, $\mathsf{FS}$ is relatively interpretable in $\mathsf{FSR}^{\emptyset} + \mathsf{TI}({<} \varphi 20)$ (see Theorem~\ref{thm:truth-def-FS}).
To prove this, we show self-realisability of $\mathsf{FSR}$ in the next section.

%Now, let $\mathsf{GCR}^{++}$ be an extension of $\mathsf{GCR}^+$ such that the reflection rule allows arbitrary $\mathcal{L}_R$-sentences.
%In contrast to $\mathsf{GCR}^+$, it is no longer clear whether Proposition~\ref{soundness_FSR} holds for $\mathsf{GCR}^{++}$, i.e., $\mathsf{GCR}^{++}$ is sound with respect to any choice of the pole, or not.

%We conclude this chapter by listing two open problems on the proof-theoretic strength.

%\begin{conjecture}\label{strength_GCR+_FSR}
%\begin{enumerate}
%\item $\mathsf{GCR}^+$ is proof-theoretically equivalent to $\mathsf{FS}$.

%\item $\mathsf{GCR}^{++}$ is consistent. In particular, it is proof-theoretically equivalent to $\mathsf{FS}$.
%\end{enumerate}
%\end{conjecture}.

%%%%%%%%%%%%%%%%%%%%%%%%%%%%%%%%%%%%%%%%%%%%

\section{Self Realisability of $\mathsf{FSR}$}
\label{sec:self-realisability}

In this section, we show that $\mathsf{FSR}$ and even $\mathsf{FSR}^{\emptyset}$ are \emph{self realisable}.
That is, for any theorem $A$ of $\mathsf{FSR}^{\emptyset}$, we can find a term $s$ such that $\mathsf{FSR} \vdash s \in \lvert A \rvert$ (Corollary~\ref{cor:real_pole-empty}). 
%Finally, conservativity of $\mathsf{FSR}$ (Proposition~\ref{conservativity_FSR}) and some open problems are mentioned.

%To begin with, we list some compositional principles for $\mathrm{T}$.

First, realisability of $\mathsf{PA}$ is an easy generalisation of the previous analysis of $\mathsf{CR}$~\cite{hayashi2024compositional}.
\begin{lemma}[cf.~{Fact~\ref{fact:number-realisability}}]\label{lem:continuation_constant_GCR}
There are numbers \( \mathsf{k}_\pi \) and \( \mathsf{k}_\pole \) such that:
\begin{enumerate}
\item %There is a number $\mathsf{k}_{\pi}$ such that: 
\(
	\mathsf{GCR} \vdash \forall \gn A , \gn B \in \pred{Sent}_{\Lr}. \ \forall a( a \F \gn A \to (\mathsf{k}_{\pi} \cdot a)\T \ulcorner A \to B \urcorner );
\)
\item %There is a number $\mathsf{k}_{\pole}$ such that: 
\(
	\mathsf{GCR} \vdash \forall a \big( a \in \pole  \to \forall \ulcorner A \urcorner \in \pred{Sent}_{\Lr}. \ (\mathsf{k}_{\pole}\cdot a)\T \ulcorner A \urcorner \big).
\)
\end{enumerate}
\end{lemma}

\begin{lemma}[Realisation of $\mathsf{PA}$]\label{explicit_realisability_PA}
For every $\Lr$-formula $A,$
if $\mathsf{PA} \vdash A,$ then there exists a term $s$ such that $\mathsf{GCR} \vdash s \in \lvert A \rvert.$
\end{lemma}
\begin{proof}
By induction on the length of the derivation of $A$,
we can give the required term in the same way as the proof of Theorem~\ref{formal_realisability_PA}. \qed
\end{proof}

Next, we give a realiser for each axiom of $\mathsf{GCR}$.
\begin{lemma}[Realisation of $(\mathrm{Ax}_{\pole})$]\label{realisability_Axpole}
There exists a term $s$ such that $\mathsf{GCR} \vdash s \in \lvert (\mathrm{Ax}_{\pole}) \rvert.$
\end{lemma}

\begin{proof}
%We show the claim by induction on the length of the derivation of $A$.
%If $A$ is obtained by an axiom or a rule of $\mathsf{PA},$ then the term is given in the same way as the proof of Lemma~\ref{}. Thus, it suffices to consider the axioms and rules of $\mathsf{RFS}.$
%\begin{description}
%\item[$(\mathsf{Ax}_{\pole})$] Assume that $A$ is the axiom $(\mathsf{Ax}_{\pole})$
Taking any $a,b$ and $c$, we want to find a term that realises the following:
\begin{equation}
	\pred{T}_1(a,b,c) \to ( \pred{U}(c) \in \pole \to \langle a, b \rangle \in \pole).
	\label{eqn:KleeneT}
\end{equation}
We show that the term $f := \lambda u. \langle (u)_0, ((u)_1)_0, ((u)_1)_1\rangle$ suffices. That is, we prove
\begin{displaymath}
f \in \lvert \pred{T}_1(a,b,c) \to ( \pred{U}(c) \in \pole \to \langle a, b \rangle \in \pole) \rvert.
\end{displaymath}
Take any term $u$ which refutes \eqref{eqn:KleeneT}:
\begin{displaymath}
u \in \lVert \pred{T}_1(a,b,c) \to ( \pred{U}(c) \in \pole \to \langle a, b \rangle \in \pole) \rVert.
\end{displaymath}
We need to show $\langle f,u \rangle \in \pole.$ For this, we prove that $\langle (u)_0, ((u)_1)_0, ((u)_1)_1\rangle \in \pole.$
Then, we have $\langle f,u \rangle \in \pole$ by the axiom $(\mathrm{Ax}_{\pole})$.

By the definition of $\lVert A \rVert$, we have:
\begin{align}
&(u)_0 \in \lvert \pred{T}_1(a,b,c) \rvert \label{real_Axpole_1}, \\
&(u)_1 \in \lVert \pred{U}(c) \in \pole \to \langle a, b \rangle \in \pole \rVert \label{real_Axpole_2}.
\end{align}
Similarly, (\ref{real_Axpole_2}) is equivalent to:
 \begin{align}
&((u)_1)_0 \in \lvert \pred{U}(c) \in \pole \rvert \label{real_Axpole_3}, \\
&((u)_1)_1 \in \lVert \langle a, b \rangle \in \pole \rVert \label{real_Axpole_4}.
\end{align} 

To verify $\langle (u)_0, ((u)_1)_0, ((u)_1)_1\rangle \in \pole$, we divide the cases by whether $\pred{T}_1(a,b,c)$ and $\pred{U}(c) \in \pole$ hold or not.
\begin{itemize}
\item If $\pred{T}_1(a,b,c)$ is not true, then we have $\langle ((u)_1)_0, ((u)_1)_1 \rangle \in \lVert \pred{T}_1(a,b,c) \rVert$ and thus by (\ref{real_Axpole_1}), we also get $\langle (u)_0, ((u)_1)_0, ((u)_1)_1 \rangle \in \pole,$ as required. 

\item Thus, we can assume that $\pred{T}_1(a,b,c)$ is true. By (\ref{real_Axpole_1}) we have $\forall w (w \in \pole \to \langle (u)_0, w \rangle \in \pole).$
In addition, if $\pred{U}(c) \in \pole$ is not true, then $\langle ((u)_1)_0, ((u)_1)_1 \rangle \in \pole$ by (\ref{real_Axpole_3}), and hence
%(\ref{}) implies 
$\langle (u)_0, ((u)_1)_0, ((u)_1)_1\rangle \in \pole.$
%$v = \langle (v)_0, (v)_1 \rangle \in \pole$.

\item Thus, we can further suppose that $\pred{U}(c) \in \pole$ is true, then we similarly have $\forall w (w \in \pole \to \langle ((u)_1)_0, w \rangle \in \pole).$ Moreover, it follows that $\langle a, b \rangle \in \pole$ by the axiom $(\mathrm{Ax}_{\pole})$.
Thus, from (\ref{real_Axpole_4}) we have $((u)_1)_1 \in \pole,$ which implies $\langle ((u)_1)_0, ((u)_1)_1 \rangle \in \pole,$ which in turn yields 
$\langle (u)_0, ((u)_1)_0, ((u)_1)_1\rangle \in \pole.$ 
\end{itemize}
Therefore, in any case, we obtain $\langle (u)_0, ((u)_1)_0, ((u)_1)_1\rangle \in \pole.$ \qed
%Note that the identity function $u$ does not depend on the choice of $a,b,$ and $c$ . Thus, the universal closure of $(\mathsf{Ax}_{\pole})$ is realised by $u' := \lambda abc. u.$
%\end{description}
\end{proof}

\begin{lemma}[Realisation of $(\mathrm{Ax}_{\T})$]\label{realisability_AxT}
There exists a term $s$ such that $\mathsf{GCR} \vdash s \in \lvert (\mathrm{Ax}_{\T}) \rvert.$
\end{lemma}
\begin{proof}
Taking any $x$, we want to find a realiser for the following:
\[
\pred{Sent}_{\Lr}(x) \to  \big( a \T x \leftrightarrow \forall b  (b \F x \to \langle a, b \rangle \in \pole) \big).
\]
If $\pred{Sent}_{\Lr}(x)$ does not hold, the above formula is derivable in $\mathsf{PA}.$ 
Thus, by Lemma~\ref{explicit_realisability_PA}, we can primitive recursively give a realiser for the formula.
Therefore, taking any code $\ulcorner A \urcorner$ of an $\Lr$-sentence, it suffices to consider the following:
\[
a \T \ulcorner A \urcorner \leftrightarrow \forall b  (b \F \ulcorner A \urcorner \to \langle a, b \rangle \in \pole).
\]
As the realisability is closed under logic, we separate the formula into the following:
\begin{enumerate}
\item $a \T \ulcorner A \urcorner \to (b \F \ulcorner A \urcorner \to \langle a, b \rangle \in \pole)$ for any fixed $b$;
\item $\forall b  (b \F \ulcorner A \urcorner \to \langle a, b \rangle \in \pole) \to a \T \ulcorner A \urcorner.$
\end{enumerate}
1. Take any $u \in \lVert a \T \ulcorner A \urcorner \to (b \F \ulcorner A \urcorner \to \langle a, b \rangle \in \pole) \rVert,$ then we have:
\begin{align}
&(u)_0 \in \lvert a \T \ulcorner A \urcorner \rvert \label{real_AxT_1}, \\
&((u)_1)_0 \in \lvert b \F \ulcorner A \urcorner \rvert \label{real_AxT_2}, \\
&((u)_1)_1 \in \lVert \langle a,b \rangle \in \pole \rVert \label{real_AxT_3}.
\end{align}
Then, we want to find a procedure $f$ such that $\langle f,u \rangle \in \pole$.

From (\ref{real_AxT_1}) and (\ref{real_AxT_2}), we have:
\begin{align}
&(u)_0 \T \ulcorner\forall b (b \in \lVert A \rVert \to \langle \dot{a}, b \rangle \in \pole ) \urcorner \label{real_AxT_1'}, \\
&((u)_1)_0 \T \ulcorner \dot{b} \in \lVert A \rVert \urcorner \label{real_AxT_2'}.
\end{align}
Thus, by Lemma~\ref{lem:partial_compositionality_GCR}, we obtain: 
\begin{align}
(\mathsf{i} \cdot \langle \mathsf{s} \cdot \langle (u)_0, b \rangle, ((u)_1)_0) \rangle ) \T \ulcorner \langle \dot{a}, \dot{b} \rangle \in \pole \urcorner \label{real_AxT_4}.
\end{align}
From (\ref{real_AxT_3}) and Lemma~\ref{explicit_formal_realisability}, %the axioms $(\mathsf{RGCT}_{\pole})$ and $(\mathsf{RGCT}_{\neg \pole})$,
 we have:
\begin{align}
&((u)_1)_1 \F \ulcorner \langle \dot{a}, \dot{b} \rangle \in \pole \urcorner \label{real_AxT_5}.
\end{align} 
Thus by (\ref{real_AxT_4}), (\ref{real_AxT_5}), and the axiom $(\mathrm{Ax}_{\T}),$ we have 
\[ \langle \mathsf{i} \cdot \langle \mathsf{s} \cdot \langle (u)_0, b \rangle, ((u)_1)_0 \rangle , ((u)_1)_1 \rangle \in \pole. \]

Now, let $f$  be such that $f \cdot b:\simeq \lambda u. \langle \mathsf{i} \cdot \langle \mathsf{s} \cdot \langle (u)_0, b \rangle, ((u)_1)_0 \rangle, ((u)_1)_1 \rangle,$ then we can show $\langle f \cdot b, u \rangle \in \pole$ with the help of $(\mathrm{Ax}_{\pole}).$ %Since this $f$ does not depend on the particular choice of $x,y,$ and $z$,
%we can employ $f$ as a required realiser. 

2. Take any $u$ such that:
\begin{align}
&(u)_0 \in \lvert \forall b  (b \F \ulcorner A \urcorner \to \langle a, b \rangle \in \pole) \rvert \label{real_AxT_6}, \\
&(u)_1 \in \lVert a \T \ulcorner A \urcorner \rVert \label{real_AxT_7}.
\end{align}
Then, we want to find a procedure $g$ such that $\langle g,u \rangle \in \pole$.

From (\ref{real_AxT_6}), we have:
\begin{align}
&\forall b, w \big( w \in \lvert b \F \ulcorner A \urcorner \rvert \to (\mathsf{i} \cdot \langle \mathsf{s} \cdot \langle (u)_0, b\rangle, w \rangle) \in \lvert \langle a,b \rangle \in \pole \rvert \big) \label{real_AxT_8}.
\end{align}

From (\ref{real_AxT_7}) we have:
\begin{align}
&(u)_1 \F \ulcorner \forall b (b \in \lVert A \rVert \to \langle \dot{a}, b\rangle \in \pole) \urcorner \label{real_AxT_9}.
\end{align}
Thus, by the axioms $(\mathrm{GCR}_{\forall})$ and $(\mathrm{GCR}_{\to})$, it follows that:
\begin{align}
&(((u)_1)_1)_0 \T \ulcorner ((\dot{u})_1)_0 \in \lVert A \rVert \urcorner \label{real_AxT_10}, \\
&(((u)_1)_1)_1 \F \ulcorner \langle \dot{a}, ((\dot{u})_1)_0\rangle \in \pole \urcorner \label{real_AxT_11}.
\end{align}
By (\ref{real_AxT_8}) and (\ref{real_AxT_10}), we obtain by taking $b:= ((u)_1)_0$ and $w:=(((u)_1)_1)_0$:
\begin{align}
&(\mathsf{i} \cdot \langle \mathsf{s} \cdot \langle (u)_0, ((u)_1)_0 \rangle, (((u)_1)_1)_0 \rangle) \in \lvert \langle a,((u)_1)_0 \rangle \in \pole \rvert \label{real_AxT_12}.
\end{align}
By (\ref{real_AxT_12}) and Lemma~\ref{explicit_formal_realisability}, %with the axioms $(\mathsf{RGCT}_{\pole})$ and $(\mathsf{RGCT}_{\neg \pole}),$ 
we also have:
\begin{align}
&(\mathsf{i} \cdot \langle \mathsf{s} \cdot \langle (u)_0, ((u)_1)_0 \rangle, (((u)_1)_1)_0 \rangle) \T \ulcorner \langle \dot{a},((\dot{u})_1)_0 \rangle \in \pole \urcorner \label{real_AxT_13}.
\end{align}
Therefore, by (\ref{real_AxT_11}) and (\ref{real_AxT_13}), we obtain a contradictory pair:
\[
\langle \mathsf{i} \cdot \langle \mathsf{s} \cdot \langle (u)_0, ((u)_1)_0 \rangle, (((u)_1)_1)_0 \rangle, (((u)_1)_1)_1 \rangle \in \pole.
\]
Thus, define $g:= \lambda u. \langle \mathsf{i} \cdot \langle \mathsf{s} \cdot \langle (u)_0, ((u)_1)_0 \rangle, (((u)_1)_1)_0 \rangle, (((u)_1)_1)_1 \rangle$. \qed
\end{proof}

\begin{lemma}[Realisation of $(\mathrm{Reg})$]\label{realisability_GCREq}
There exists a term $s$ such that $\mathsf{GCR} \vdash s \in \lvert (\mathsf{Reg}) \rvert.$
\end{lemma}
\begin{proof}
Here, the exact form of the formula $\pred{Eq}(y,z)$ is required.
Since $\pred{Eq}(y,z)$ is $\Sigma_1$-formula, we shall rephase it in the form $\exists x \pred{Eq}'(x,y,z)$, where $\pred{Eq}'(x,y,z)$ is some atomic predicate of $\mathcal{L}$.

%We consider the case $\mathrm{P} \equiv \mathrm{F},$ since the case $\mathrm{P} \equiv \mathrm{T}$ is similar.
Take any $x \in \mathbb{N},$ any $\ulcorner s = t \urcorner, \ulcorner A(0) \urcorner \in \pred{Sent}_{\Lr},$ and any $a \in \mathbb{N}.$ Then, similar to the proof of Lemma~\ref{realisability_AxT}, it suffices to find a realiser of the following:
\[
 \pred{Eq}'(x, \ulcorner s \urcorner,\ulcorner t \urcorner) \to (a \F \ulcorner A(s) \urcorner \to a \F \ulcorner A(t) \urcorner).
\]
Thus, we take any $u$ such that:
\begin{align}
&(u)_0 \in \lvert \pred{Eq}'(x, \ulcorner s \urcorner,\ulcorner t \urcorner) \rvert \label{real_GCR_Eq_1}, \\
&((u)_1)_0 \in \lvert a \F \ulcorner A(s) \urcorner \rvert \label{real_GCR_Eq_2}, \\
&((u)_1)_1 \in \lVert a \F \ulcorner A(t) \urcorner \rVert \label{real_GCR_Eq_3}. 
\end{align} 
We will show that $f:= \lambda u. \langle (u)_0,((u)_1)_0,((u)_1)_1 \rangle$ satisfies $\langle f, u \rangle \in \pole.$
So, by the axiom $(\mathrm{Ax}_{\pole}),$ it is enough to prove $\langle (u)_0,((u)_1)_0,((u)_1)_1 \rangle \in \pole.$

If the atomic $\mathcal{L}$-formula $\pred{Eq}'(x, \ulcorner s \urcorner,\ulcorner t \urcorner)$ is false, then by (\ref{real_GCR_Eq_1}), $(u)_0$ contradics any number, that is, we have $\langle (u)_0,((u)_1)_0,((u)_1)_1 \rangle \in \pole,$ as required.
Therefore, we assume $\pred{Eq}'(x, \ulcorner s \urcorner,\ulcorner t \urcorner)$ is true.
Hence, (\ref{real_GCR_Eq_1}) implies that in order to show $\langle (u)_0,((u)_1)_0,((u)_1)_1 \rangle \in \pole,$ 
we need to prove $\langle ((u)_1)_0,((u)_1)_1 \rangle \in \pole.$

By (\ref{real_GCR_Eq_2}) and (\ref{real_GCR_Eq_3}), we respectively obtain:
\begin{align}
&((u)_1)_0 \T \ulcorner \dot{a} \in \lVert A(s) \rVert \urcorner \label{real_GCR_Eq_4}, \\
&((u)_1)_1 \F \ulcorner \dot{a} \in \lVert A(t) \rVert \urcorner \label{real_GCR_Eq_5}.
\end{align}
Since we assumed $\pred{Eq}'(x, \ulcorner s \urcorner,\ulcorner t \urcorner)$, by (\ref{real_GCR_Eq_5}) and the axiom $\mathrm{Reg},$ we obtain:
\begin{align}
&((u)_1)_1 \F \ulcorner \dot{a} \in \lVert A(s) \rVert \urcorner \label{real_GCR_Eq_6}.
\end{align} 
Therefore, from (\ref{real_GCR_Eq_4}), (\ref{real_GCR_Eq_6}), and the axiom $(\mathrm{GCR}_{\T}),$ it follows that 
\[\langle ((u)_1)_0,((u)_1)_1 \rangle \in \pole,\] which implies $\langle (u)_0,((u)_1)_0,((u)_1)_1 \rangle \in \pole,$ as desired. \qed
\end{proof}

\begin{lemma}[Realisation of $(\mathrm{GCR}_{P1})$]\label{realisability_GCRnegP}
There exists a term $s$ such that $\mathsf{GCR} \vdash s \in \lvert (\mathrm{GCR}_{P1}) \rvert.$
\end{lemma}
\begin{proof}
Take any $u$ such that:
\begin{align}
&(u)_0 \in \lvert \neg P \vec{x} \rvert \label{real_GCR_negP_1}, \\
&(u)_1 \in \lVert a \F \ulcorner P \dot{\vec{x}} \urcorner \rVert \label{real_GCR_negP_2}.
\end{align} 
In order to show that $f:=\lambda u. \langle (u)_0, (u)_1 \rangle$ satisfies $\langle f,u \rangle\in \pole,$
we prove that $\langle (u)_0, (u)_1 \rangle \in \pole$.
Statement 
(\ref{real_GCR_negP_1}) is equivalent to:
\begin{align}
\forall w \Big( (w)_0 \in \lvert P \vec{x} \rvert \to \big( (w)_1 \in \lVert \bot \rVert \to \langle (u)_0,w \rangle \in \pole \big) \Big) \label{real_GCR_negP_3},
\end{align}
where recall  $\neg A$ is defined as $ A \to \bot$ and $\bot$ is some false equation. 
From (\ref{real_GCR_negP_2}), Lemma~\ref{explicit_formal_realisability}, and the axiom $(\mathrm{GCR}_{\to}),$ we deduce that
\begin{align}
(u)_1 \in \lVert a \F \ulcorner P \dot{\vec{x}} \urcorner \rVert &\Leftrightarrow (u)_1 \F \ulcorner P \dot{\vec{x}} \to \dot{a} \in \pole \urcorner \notag \\
&\Leftrightarrow ((u)_1)_0 \T \ulcorner P \dot{\vec{x}} \urcorner \land ((u)_1)_1 \F \ulcorner \dot{a} \in \pole \urcorner \notag \\
&\Rightarrow ((u)_1)_0 \in \lvert P \vec{x} \rvert \land ((u)_1)_1 \in \lVert \bot \rVert \label{real_GCR_negP_4}.
\end{align} 
Thus, by (\ref{real_GCR_negP_3}) and (\ref{real_GCR_negP_4}), we have $\langle (u)_0, (u)_1 \rangle \in \pole.$ \qed
%Therefore, $f := \lambda u.u$ is a required function.
\end{proof}

\begin{lemma}[Realisation of $(\mathrm{GCR}_{P2})$]\label{realisability_GCRP}
There exists a term $s$ such that $\mathsf{GCR} \vdash s \in \lvert (\mathrm{GCR}_{P2}) \rvert.$
\end{lemma}
\begin{proof}
The proof is similar to for $(\mathrm{GCR}_{P1})$. \qed
\end{proof}

\begin{lemma}[Realisation of $(\mathrm{GCR}_{\to})$]\label{realisability_GCRto}
For any number $a$ and any code $\ulcorner A\to B \urcorner$ of an $\Lr$-sentence, all the following are realisable in $\mathsf{GCR}$:
\begin{enumerate}
\item $a \F \ulcorner A \to B \urcorner \to (a)_0 \T \ulcorner A \urcorner,$
\item $a \F \ulcorner A \to B \urcorner \to (a)_1 \F \ulcorner B \urcorner,$
\item $(a)_0 \T \ulcorner A \urcorner \to \big( (a)_1 \F \ulcorner B \urcorner \to a \F \ulcorner A \to B \urcorner \big)$.
\end{enumerate}
Therefore, we can give a term that realises $(\mathrm{GCR}_{\to})$.
\end{lemma}
\begin{proof}
\begin{enumerate}
\item Take any $u$ such that $(u)_0 \in \lvert a \F \ulcorner A \to B \urcorner \rvert$ and $(u)_1 \in \lVert (a)_0 \T \ulcorner A \urcorner \rVert,$
then we have $(u)_0 \T \ulcorner \dot{a} \in \lVert A \to B \rVert \urcorner$ and $(u)_1 \F \ulcorner (\dot{a})_0 \in \lvert A \rvert \urcorner.$ By Lemma~\ref{formal_realisability_PA}, we obtain a number $w$ such that $w \T \ulcorner (\dot{a})_0 \in \lvert A \rvert \land (\dot{a})_1 \in \lVert B \rVert \to (\dot{a})_0 \in \lvert A \rvert \urcorner,$ 
from which and $(u)_0 \T \ulcorner (\dot{a})_0 \in \lvert A \rvert \land (\dot{a})_1 \in \lVert B \rVert \urcorner,$ 
we get $(\mathsf{i} \cdot \langle w, (u)_0 \rangle) \T \ulcorner (\dot{a})_0 \in \lvert A \rvert \urcorner.$ Thus, $\langle\langle \mathsf{i} \cdot \langle w, (u)_0 \rangle, (u)_1\rangle \in \pole$ by the axiom $(\mathrm{Ax}_{\T}).$
Therefore, the term $f:= \lambda u. \langle\langle \mathsf{i} \cdot \langle w, (u)_0 \rangle, (u)_1 \rangle$ satisfies $\langle f,u \rangle \in \pole$.
\item Similarly to the proof of the item 1.
\item Take any $u$ such that:
\begin{align}
&(u)_0 \T \ulcorner (\dot{a})_0 \in \lvert A \rvert \urcorner \label{real_GCR_to_1}, \\
&((u)_1)_0 \T \ulcorner (\dot{a})_1 \in \lVert B \rVert \urcorner \label{real_GCR_to_2}, \\
&((u)_1)_1 \F \ulcorner \dot{a} \in \lVert A \to B \rVert \urcorner \label{real_GCR_to_3}.
\end{align}
By Lemma~\ref{explicit_realisability_PA} and $\mathrm{NECR}$, we obtain a number $w$ such that
\begin{align}
&w \T \ulcorner (\dot{a})_0 \in |A| \to \Big( (\dot{a})_1 \in \lVert B \rVert \to \big( (\dot{a})_0 \in \lvert A \rvert \land (\dot{a})_1 \in \lVert B \rVert \big) \Big) \urcorner\label{real_GCR_to_4}.
\end{align}
Thus, from (\ref{real_GCR_to_1}) and (\ref{real_GCR_to_2}), we obtain:
\begin{align}
(\mathsf{i} \cdot \langle \mathsf{i} \cdot \langle w, (u)_0 \rangle, ((u)_1)_0 \rangle) \T \ulcorner (\dot{a})_0 \in \lvert A \rvert \land (\dot{a})_1 \in \lVert B \rVert \urcorner \label{real_GCR_to_5}.
\end{align}
On the other hand, (\ref{real_GCR_to_3}) is equivalent to:
\begin{align}
&((u)_1)_1 \F \ulcorner (\dot{a})_0 \in \lvert A \rvert \land (\dot{a})_1 \in \lVert B \rVert \urcorner \label{real_GCR_to_6}.
\end{align}
Thus, we have $\langle \mathsf{i} \cdot \langle \mathsf{i} \cdot \langle w, (u)_0 \rangle, ((u)_1)_0 \rangle, ((u)_1)_1 \rangle \in \pole$.
Therefore, we can take $g := \lambda u. \langle \mathsf{i} \cdot \langle \mathsf{i} \cdot \langle w, (u)_0 \rangle, ((u)_1)_0 \rangle, ((u)_1)_1 \rangle$ as a required realiser. \qed
\end{enumerate}
\end{proof}

\begin{lemma}[Realisation of $(\mathrm{GCR}_{\forall})$]\label{realisability_GCRforall}
There exists a term $s$ such that $\mathsf{GCR} \vdash s \in \lvert (\mathrm{GCR}_{\forall}) \rvert.$
\end{lemma}
\begin{proof}
The proof is similar to for $(\mathrm{GCR}_{\to})$. \qed
\end{proof}

\begin{theorem}\label{explicit_realisability_FSR}
For each $\mathcal{L}_R$-formula $A$, if $\mathsf{FSR} \vdash A,$
then there exists a term $s$ such that $\mathsf{FSR} \vdash s \in \lvert A \rvert.$
\end{theorem}
\begin{proof}
The proof is by induction on the derivation length of $A$.
Assume that $\mathsf{FSR} \vdash A.$
If $A$ is obtained by the axioms of $\mathsf{GCR},$ then we already have an answer by the above lemmas.
Thus, the remaining cases are $(\mathrm{NECR})$ and $(\mathrm{CONECR})$. 
\begin{description}
\item[$(\mathrm{NECR})$] Assume that $A$ is of the form $ s \T \ulcorner B \urcorner$ and is derived from $s \in \lvert B \rvert.$
Then, by induction hypothesis, we have some term $t$ such that $\mathsf{FSR} \vdash t \in \lvert s \in \lvert B \rvert \rvert.$
Thus, by $(\mathrm{NECR})$ we deduce $t \T \ulcorner \dot{s} \in \lvert B \rvert \urcorner,$ which is $t \in \lvert s \T \ulcorner B \urcorner \rvert$, and hence $t$ is a required term.

\item[$(\mathrm{CONECR})$] Assume that $A$ is of the form $ s \in \lvert B \rvert$ and is derived from $s \T \ulcorner B \urcorner$.
Then, by induction hypothesis, we have some term $t$ such that $\mathsf{FSR} \vdash t \in \lvert s \T \ulcorner B \urcorner \rvert.$ Thus, by $(\mathrm{CONECR})$ we deduce $t \in \lvert s \in \lvert B \rvert \rvert,$ and hence, $t$ is a required term. \qed
\end{description}
\end{proof}

\begin{corollary}\label{cor:real_pole-empty}
There exists a term $t$ such that $\mathsf{FSR} \vdash t \in \lvert \pole = \emptyset \rvert$.
Thus, every theorem of $\mathsf{FSR}^{\emptyset}$ is realisable in $\mathsf{FSR}$, and hence $\mathsf{FSR}^{\emptyset}$ is self realisable.
\end{corollary}

\begin{proof}
The claim $t \in \lvert \pole = \emptyset \rvert$ is equivalent to the following:
\begin{align}
t \in \lvert \pole = \emptyset \rvert &\Leftrightarrow \forall x. \ x \in \lVert \forall y. \ y \in \pole \to \bot \rVert \to \langle t, x \rangle \in \pole \notag \\
&\Leftrightarrow \forall x. \ (x)_1 \in \lVert (x)_0 \in \pole \to \bot \rVert \to \langle t,x \rangle \in \pole \notag \\
&\Leftrightarrow \forall x. \ ((x)_1)_0 \in \lvert (x)_0 \in \pole \rvert \land ((x)_1)_1 \in \lVert \bot \rVert \to \langle t,x \rangle \in \pole \notag \\
&\Leftrightarrow \forall x. \ ((x)_1)_0 \in \lvert (x)_0 \in \pole \rvert \to \langle t,x \rangle \in \pole. \label{lem:real_pole-empty:claim}
\end{align}
So, to prove (\ref{lem:real_pole-empty:claim}), we take any $x$ such that $((x)_1)_0 \in \lvert (x)_0 \in \pole \rvert$, then we want to find a term $t$ such that $\langle t,x \rangle \in \pole$.
We divide the cases by the truth value of $(x)_0 \in \pole$:
\begin{itemize}
\item If $(x)_0 \in \pole$, then $((x)_1)_0 \in \lvert (x)_0 \in \pole \rvert$ implies $\langle ((x)_1)_0, (x)_0 \rangle \in \pole$ by Lemma~\ref{lem:truth-condition-P}.
\item Otherwise, it follows that $\forall z. \ \langle ((x)_1)_0, z \rangle \in \pole$ by Lemma~\ref{lem:truth-condition-P}, and hence we again get $\langle ((x)_1)_0, (x)_0 \rangle \in \pole$.
\end{itemize}
Thus, for the term $t := \lambda x. \langle ((x)_1)_0, (x)_0 \rangle$, we have $\langle t,x \rangle \in \pole$ by the axiom ($\mathsf{Ax}_{\pole}$).
Therefore, (\ref{lem:real_pole-empty:claim}) is proved. \qed
\end{proof}

%%%%%%%%%%%%%%%%%%%%%%%%%%%%%%%%%%%%%%%%%%%%%%

\section{Applications of Self Realisability}\label{sec:applications}

%\subsection{Truth as a Special Case of Classical Realisability}

%Note also that $\mathcal{T}_{\mathsf{FS}}(\mathrm{T} (t)) : \equiv \forall x. \ x \mathrm{F} (\tau (t)) \to \langle 0, x \rangle \in \pole$ is clearly equivalent to $0 \mathrm{T}(\tau (t))$ by the axiom $\mathsf{Ax}_{\mathrm{T}}$. 
%Although we could adopt the latter formula as the definition of $\mathcal{T}_{\mathsf{FS}}(\mathrm{T} (t))$, it is convenient to consider a translation $\mathcal{T}_{\mathsf{FS}}$ that does not contain the realisation predicate $x \mathrm{T} y$ for the proof of Proposition~\ref{inconsistency_GCR}.

In this section, we provide two applications of self realisability of $\mathsf{FSR}^{(\emptyset)}$.
Firstly, relative truth definablity for $\mathsf{GCT}$ (Proposition~\ref{prop:truth-def-GCT}) can be generalised to that for $\mathsf{FS}$.
Consequently, the proof-theoretic lower bound of $\mathsf{FSR}^+$ is determined.

\begin{theorem}\label{thm:truth-def-FS}
\begin{enumerate}
\item 
Let $\mathcal{T}_{\mathsf{FS}}$ be as in Proposition~\ref{prop:truth-def-GCT}.
%$\mathsf{GCT} + \mathsf{NEC}$ is relatively truth definable in $\mathsf{FSR}^{\emptyset}$ by the translation $\mathcal{T}_{\mathsf{FS}}$.
For any $\mathcal{L}_T$-formula $A$, if $\mathsf{GCT} + \mathsf{NEC} \vdash A$, then $\mathsf{FSR}^{\emptyset} \vdash \mathcal{T}_{\mathsf{FS}}(A)$.

\item $\mathsf{FS}$ is relatively truth definable in $\mathsf{FSR}^{\emptyset} + \mathsf{TI}({<} \varphi20)$ by the same translation.

\end{enumerate}
\end{theorem}

\begin{proof}
1. The proof proceeds by induction on the derivation of $A$.
By Proposition~\ref{prop:truth-def-GCT}, it suffices to consider the case of $\mathrm{NEC}$, so we assume that $A$ is of the form $ \T \ulcorner B \urcorner$ and is derived from $B$.
By the induction hypothesis, we get $\mathsf{FSR}^{\emptyset} \vdash \mathcal{T}_{\mathsf{FS}}(B)$.
Since $\mathsf{FSR}^{\emptyset}$ is self realisable by Corollary~\ref{cor:real_pole-empty}, there exists a term $s$ such that $\mathsf{FSR}^{\emptyset} \vdash s \in \lvert \mathcal{T}_{\mathsf{FS}}(B) \rvert$.
Because of $\pole = \emptyset$, we also have $0 \in \lvert \mathcal{T}_{\mathsf{FS}}(B) \rvert$.
By ($\mathrm{NECR}$), it follows that $\mathsf{FSR}^{\emptyset} \vdash 0 \T \ulcorner \mathcal{T}_{\mathsf{FS}}(B) \urcorner$.
Here, $0 \T \ulcorner \mathcal{T}_{\mathsf{FS}}(B) \urcorner$ is equivalent to $ \mathcal{T}_{\mathsf{FS}}(\T \ulcorner B \urcorner)$, so we are done.

2. By Fact~\ref{fact:FS_GCT-NEC-TI} and item 1 of this theorem, it is sufficient to see that $\mathsf{FSR}^{\emptyset} + \mathsf{TI}({<} \varphi20)$ derives $\mathcal{T}_{\mathsf{FS}}(A)$ for each instance $A$ of $\mathrm{TI}({<} \varphi20)$ in the language $\Lt$.
But, by the definition of $\mathcal{T}_{\mathsf{FS}}$, the formula $\mathcal{T}_{\mathsf{FS}}(A)$ is again an instance of $\mathrm{TI}({<} \varphi20)$ in the language $\Lr$, which is an axiom of $\mathsf{FSR}^{\emptyset} + \mathsf{TI}({<} \varphi20)$. \qed
\end{proof}

%\begin{lemma}\label{lem:realisation-FS}
%Take any $\mathcal{L}_{\mathrm{T}}$-formula $A$.
%If $\mathsf{FS} \vdash A$, then there exists a closed term $t$ such that $\mathsf{FSR} \vdash t \mathrm{r} A$.
%\end{lemma}

\begin{theorem}\label{thm:strength_FSR+}
Let $\mathsf{S} \equiv \mathsf{T}$ mean that the systems $\mathsf{S}$ and $\mathsf{T}$ have the same $\mathcal{L}$-theorems.
Then, the following holds:
\[
\mathsf{FS}\equiv \mathsf{PA} + \mathsf{TI}({<} \varphi20) \equiv \mathsf{FSR}^{+} \equiv \mathsf{FSR}^{\emptyset} + \mathsf{TI}({<} \varphi20) .
\]
\end{theorem}

\begin{proof}
Firstly, we remark that Halbach (\cite[Theorem~5.13]{halbach1994system}) showed that $\mathsf{FS} \equiv \mathsf{GCT} + \mathsf{NEC} \equiv \mathsf{PA} + \mathsf{TI}({<} \varphi20)$.
Secondly, Proposition~\ref{upper-bound_FSR-empty} establishes that every $\mathcal{L}$-consequence of $\mathsf{FSR}^{\emptyset} + \mathsf{TI}({<} \varphi 20)$ is derivable in $\mathsf{FS}$.
Since $\mathsf{FSR}^+$ is a subtheory of $\mathsf{FSR}^{\emptyset}$, it suffices to show that $\mathsf{FSR}^+$ derives every $\mathcal{L}$-consequence of $\mathsf{GCT} + \mathsf{NEC}$.

By Theorem~\ref{thm:truth-def-FS}, every $\mathcal{L}$-consequence of $\mathsf{GCT} + \mathsf{NEC}$ is derivable in $\mathsf{FSR}^{\emptyset}$.
Moreover, by Corollary~\ref{cor:real_pole-empty}, $\mathsf{FSR}^{\emptyset}$ is realisable in $\mathsf{FSR}$.
Therefore, every $\mathcal{L}$-consequence of $\mathsf{GCT} + \mathsf{NEC}$ is realisable in $\mathsf{FSR}$, and thus is derived by the reflection rule in $\mathsf{FSR}^+$.
In conclusion, all of the above equivalence relations are justified. \qed

%For the upper-bound of $\mathsf{FSR}^{\emptyset} + \mathsf{TI}({<} \varphi20)$, note that the consistency proof of $\mathsf{FSR}^{\emptyset}$ in Section~\ref{subsec:model_FSR} is formalisable in $\mathsf{RT}_{< \omega}$, the ramified truth theory up to $\omega$ (see cf.~\cite[Definition~9.2]{halbach2014axiomatic}), almost in the same way as Halbach's upper-bound proof of $\mathsf{FS}$ (\cite[Theorem~5.9]{halbach1994system}). 
%Since it is well known that $\mathsf{RT}_{< \omega}$ is arithmetically as strong as $\mathsf{PA} + \mathsf{TI}({<} \varphi20)$, all the above equivalence relations are obtained. 
\end{proof}

%%%%%%%%%%%%%%%%%%%%%%%%%%%%%%%%%

\subsection{McGee's Theorem in $\mathsf{FSR}$}

The second application of self realisability concerns $\omega$-inconsistency phenomenon by McGee (\cite{mcgee1985truthlike}).
Following Halbach's formulation (\cite[Theorem~13.9]{halbach2014axiomatic}), we define $f(x,y)$ as a binary primitive function symbol such that for every $\Lt$-sentence $A$:
\begin{align}
f(0, \ulcorner A \urcorner) &= \ulcorner A \urcorner, \notag \\
f(n, \ulcorner A \urcorner) &= \ulcorner \underbrace{\T \ulcorner \T \ulcorner \cdots \T}_\text{$n$ occurrences} \ulcorner A \urcorner \urcorner ^{\cdots} \urcorner \ \ \text{for $n \geq 1$.} \notag
\end{align}
We define an $\Lt$-sentence $\sigma$ such that $\sigma \leftrightarrow \neg \forall x \big( \T (f(x, \ulcorner \sigma \urcorner) )\big)$ is true; thus $\sigma$ says that the result of prefixing the truth predicate to $\sigma$ some times is false.
Then, the following holds:

\begin{fact}[\cite{mcgee1985truthlike}]\label{fact:McGee-FSR_emp}
Let $\sigma$ be as above. 
Then, 
\begin{enumerate}
	\item $\mathsf{GCT} + \mathsf{NEC} \vdash \sigma$,
	\item $\mathsf{GCT} + \mathsf{NEC} \vdash \T(f(n, \ulcorner \sigma \urcorner))$ for each natural number $n$.
\end{enumerate}
Thus, $\mathsf{GCT} + \mathsf{NEC}$ is $\omega$-inconsistent.
\end{fact}

Since $\omega$-inconsistency is preserved by relative truth definitions (\cite[Proposition~28]{fujimoto2010relative}), it follows, by Theorem~\ref{thm:truth-def-FS} and Fact~\ref{fact:McGee-FSR_emp}, that $\mathsf{FSR}^{\emptyset}$ is $\omega$-inconsistent.
%In addition, although $\mathsf{FSR}$ itself is not $\omega$-inconsistent, $\mathcal{T}_{\mathsf{FS}}(\sigma)$ and $\mathcal{T}_{\mathsf{FS}}\big(\T(f(n, \ulcorner \sigma \urcorner))\big)$ for each $n$ are realisable in $\mathsf{FSR}$ by Theorem~\ref{thm:truth-def-FS}.
%In this sense, $\mathsf{FSR}$ \emph{almost} encompasses $\omega$-inconsistency.
%Nevertheless, with the binary predicate $x \T y$ of $\Lr$, we can form a more computationally meaningful sentence $\gamma$, using which we want to observe the (almost) $\omega$-inconsistency phenomenon in $\mathsf{FSR}$ (Corollary~\ref{cor:McGee-FSR}). 
Although $\mathsf{FSR}$ itself is not $\omega$-inconsistent by Proposition~\ref{conservativity_FSR},
$\mathsf{FSR}$, as we will show, can be seen as \emph{almost} $\omega$-inconsistent in a sense.
For this purpose, we define a computational counterpart $\gamma$ of the above sentence $\sigma$, and prove a McGee-like result for this $\gamma$.

Let $f'(x,y,z)$ be a ternary primitive function symbol such that the following hold for every $\Lr$-sentence $A$:
\begin{align}
f'(g, 0, \ulcorner A \urcorner) &= \ulcorner A \urcorner, \notag \\
f'(g, n, \ulcorner A \urcorner) &= \ulcorner (g \cdot (n-1)){\underbrace{\T \ulcorner (g \cdot (n-2))\T \ulcorner \cdots (g \cdot 0) \T}_\text{$n$ occurrences of $\T$}} \ulcorner A \urcorner \urcorner ^{\cdots} \urcorner \ \ \text{for $n \geq 1$.} \notag
\end{align}
Let $ g _x \T y$ abbreviate the $\mathcal{L}_R$-formula $(g \cdot x) \T(f'(g,x,y))$.
If $ g _x \T y$ holds, we shall say that the sequence $ \langle g \cdot x, g \cdot (x-1), \dots, g \cdot 0 \rangle $ \emph{sequentially} realises $y$.

By the Diagonal lemma, we define a sentence $\gamma$ such that:
\[
\mathsf{PA} \vdash \gamma \leftrightarrow \neg \exists g \forall x ( g _x \T \ulcorner \gamma \urcorner ).  
%\gamma \letfrightarrow \neg \forall g. \forall x. \ g \cdot x \mathrm{T} \ulcorner g \cdot (x-1) \mathrm{T} \ulcorner g \cdot (x-2) \mathrm{T} \ulcorner \dots \mathrm{T} g \cdot 0 \mathrm{T} \ulcorner \gamma \urcorner \dots \urcorner.  
\]
Hence, $\gamma$ informally says that there is no infinite sequent $\langle g \cdot 0, g \cdot 1, \dots \rangle$ such that every initial segment in reverse order sequentially realises $\gamma$.
%$\gamma$, $(g \cdot 0) \mathrm{T} \ulcorner \gamma \urcorner$, $(g \cdot 1) \mathrm{T} \ulcorner (g \cdot 0) \mathrm{T} \ulcorner \gamma \urcorner \urcorner$, $\dots$ is realisable in $\mathsf{FSR}$.
Then, our aim is to prove that $\mathsf{FSR}$ realises $\gamma$, in a way that there is an enumeration $g$ such that $\vec{g_n} \T \ulcorner \gamma \urcorner$ is derivable for each numeral $n$.
Moreover, to derive $\gamma$ itself, the assumption $\pole = \emptyset$ is shown to be necessary and sufficient over $\mathsf{FSR}$, and therefore $\omega$-inconsistency by means of $\gamma$ occurs in $\mathsf{FSR}^{\emptyset}$.
%\leigh{L:$\pole = \emptyset$ is clearly sufficient but is it ``required''?}

%\hayashi{Hayashi: Thank you for the question! $\gamma$ seems equivalent to $\pole = \emptyset$. I've slightly changed the above explanation and the statement of item~3 of Proposition~12, and I've added the proof below.}

\begin{proposition}\label{almost-omega-inconsistency}
Let $\gamma$ be as above.
\begin{enumerate}
\item There exists a term $s$ such that $\mathsf{FSR} \vdash s \in \lvert \gamma \rvert$.

\item There exists a numeral $g$ such that $\mathsf{FSR} \vdash g_n T \ulcorner \gamma \urcorner$ for each $n$.

\item $\mathsf{FSR} \vdash \gamma \leftrightarrow \pole = \emptyset$. Thus, $\mathsf{FSR}^{\emptyset}$ is $\omega$-inconsistent.
\end{enumerate}
\end{proposition}

Before proving the proposition, let us observe an immediate consequence.
We denote by $\mathsf{Ref}(\mathsf{FSR})$ the system $\mathsf{FSR}$ augmented with the \emph{reflection principle for} $\mathsf{FSR}$:
\[
\forall x \big( \mathrm{Bew}_{\mathsf{FSR}}\ulcorner A(\dot{x}) \urcorner \to A(x) \big), \ \text{for each $A(x) \in \Lr$,} 
\]
where $\mathrm{Bew}_{\mathsf{FSR}}(x)$ is a canonical derivability predicate for $\mathsf{FSR}$.

Due to $\omega$-inconsistency of $\mathsf{FS}$, it is known that the result of extending $\mathsf{FS}$ with the reflection principle for $\mathsf{FS}$ is inconsistent (cf.~\cite[Corollary~14.39]{halbach2014axiomatic}). 
In the case of $\mathsf{FSR}$, we obtain instead non-emptiness of pole.

\begin{corollary}\label{cor:McGee-FSR}
There exists a closed term $t$ such that 
$\mathsf{Ref}(\mathsf{FSR}) \vdash t \in \lvert  \bot \rvert$.
Thus, $\mathsf{Ref}(\mathsf{FSR}) \vdash \pole \neq \emptyset$.
\end{corollary}

\begin{proof}

By item~2 of Proposition~\ref{almost-omega-inconsistency}, $\mathsf{FSR}$ derives $g \cdot (n+1) \in \lvert g _n \T \ulcorner \gamma \urcorner \rvert$ for each $n$, which is formalisable by the reflection principle: $\mathsf{Ref}(\mathsf{FSR}) \vdash \forall x \big( g \cdot (x+1) \in \lvert g _x \T \ulcorner \gamma \urcorner \rvert \big)$.
Hence, by Lemma~\ref{lem:partial_compositionality_GCR}, we get a term $s$ such that $\mathsf{Ref}(\mathsf{FSR}) \vdash s \in \lvert \forall x ( g _x \T \ulcorner \gamma \urcorner ) \rvert$.
Thus, again by Lemma~\ref{lem:partial_compositionality_GCR}, we also have a term $s'$ such that $\mathsf{Ref}(\mathsf{FSR}) \vdash s' \in \lvert \exists g \forall x ( g _x \T \ulcorner \gamma \urcorner ) \rvert$.
To summarise, we get realisers for $\gamma$ and $\neg \gamma$ respectively, and hence $\bot$ is realisable in $\mathsf{Ref}(\mathsf{FSR})$. This means that the pole must be non-empty. \qed
\end{proof}

For the proof of Proposition~\ref{almost-omega-inconsistency}, we first want to find a realiser of $\gamma$.
By the definition, $\neg \gamma$ implies $\exists g \big( (g \cdot 0) \T \ulcorner \gamma \urcorner \big)$; thus, we have:
\[
\mathsf{PA} \vdash \neg \gamma \to \exists y (y \T \ulcorner \gamma \urcorner),
\]
which implies, by Lemma~\ref{explicit_realisability_PA}, that $\neg \gamma \to \exists y (y \T \ulcorner \gamma \urcorner)$ is realisable in $\mathsf{FSR}$.
Thus, if we can also prove that $\exists y (y \T \ulcorner \gamma \urcorner) \to \gamma$ is realisable in $\mathsf{FSR}$, then the realisability of $\gamma$ itself follows, as required.
The following is a subsidiary lemma:

\begin{lemma}\label{lem:real_bot}
In $\mathsf{GCR}$, the following are derivable.
\begin{enumerate}

\item $\lambda a. \langle (a)_0, (a)_1 \rangle \in \lvert x \in \lVert \bot \rVert \rvert$.

\item $x \in \lVert x \in \pole \rVert$.

\item $x \in \lvert y \T \ulcorner \bot  \urcorner \rvert \to \lambda b. \langle x,0,\lambda a. \langle (a)_0, (a)_1 \rangle,y,0 \rangle \in \lvert \bot \rvert$.
\end{enumerate}
\end{lemma}

\begin{proof}
Starting with 1, let \( f = \lambda a. \langle (a)_0, (a)_1 \rangle \). Then
\begin{align*}
 f \in \lvert x \in \lVert \bot \rVert \rvert &\Leftrightarrow f \in \lvert \bot \to x \in \pole \rvert \\
&\Leftrightarrow \forall y \big( y \in \lVert \bot \to x \in \pole \rVert \to \langle f, y \rangle \in \pole \big) \\
&\Leftrightarrow \forall y \big( (y)_0 \in \lvert \bot \rvert \land (y)_1 \in \lVert x \in \pole \rVert \to \langle f, y \rangle \in \pole \big). \\
&\Leftarrow \forall y \big( (y)_0 \in \lvert \bot \rvert \land (y)_1 \in \lVert x \in \pole \rVert \to \langle (y)_0, (y)_1 \rangle \in \pole \big). \tag{$\because$ $\mathrm{Ax}_{\pole}$}
\end{align*}
Thus, taking any $y$ such that $ (y)_0 \in \lvert \bot \rvert \land (y)_1 \in \lVert x \in \pole \rVert$, we show $\langle (y)_0, (y)_1 \rangle \in \pole$.
But, $(y)_0 \in \lvert \bot \rvert$ implies that $\langle (y)_0, (y)_1 \rangle \in \pole$, and hence 
the claim is established.
%$\langle \lambda a.a, y \rangle \in \pole$ follows by the axiom $(\mathsf{Ax}_{\pole})$.

For 2, if $x \in \pole$, then $x \in \pole$ is refuted only by the members of $\pole$, but we have just assumed $x \in \pole$. Therefore, $x \in \lVert x \in \pole \rVert$.
If $x \notin \pole$, then $x \in \pole$ is refuted by any natural number, thus $x \in \lVert x \in \pole \rVert$ is obtained again.

Finally, we show 3 by proving that $\langle x,0,f,y,0 \rangle \in \pole$ whenever $x \in \lvert y \T \ulcorner \bot  \urcorner \rvert$ where \( f \) is as in 1.
Firstly, $x \in \lvert y \T \ulcorner \bot  \urcorner \rvert$  is equivalent to:
\begin{align}
%& x \in \lvert y \mathrm{T} \ulcorner \bot  \urcorner \rvert \notag \\
& x \T \ulcorner \dot{y} \in \lvert \bot \rvert \urcorner \tag{$\because$ lem.~\ref{explicit_formal_realisability}} \\
&\Leftrightarrow x \in \lvert y \in \lvert \bot \rvert  \rvert  \tag{$\because$ lem.~\ref{explicit_formal_realisability}} \\
&\Leftrightarrow \forall z \big( z \in \lVert y \in \lvert \bot \rvert \rVert  \to \langle x,z \rangle \in \pole \big) \notag \\
&\Leftrightarrow \forall z \big( z \in \lVert \forall w. \ w \in \lVert \bot \rVert \to \langle y,w \rangle \in \pole \rVert  \to \langle x,z \rangle \in \pole \big) \notag \\
&\Leftrightarrow \forall z \big( (z)_1 \in \lVert (z)_0 \in \lVert \bot \rVert \to \langle y,(z)_0 \rangle \in \pole \rVert  \to \langle x,z \rangle \in \pole \big) \notag \\
%&\Leftrightarrow \forall z. \ (z)_1 \in \lVert (z)_0 \in \lVert \bot \rVert \to \langle y,(z)_0 \rangle \in \pole \rVert  \to \langle x,z \rangle \in \pole \notag \\
&\Leftrightarrow \forall z \big( ((z)_1)_0 \in \lvert (z)_0 \in \lVert \bot \rVert \rvert  \land ((z)_1)_1 \in \lVert \langle y,(z)_0 \rangle \in \pole \rVert  \to \langle x,z \rangle \in \pole \big). \label{lem:real_bot:assump} \notag \\ 
%\begin{split}
%& \ \ \ \ \ \ \ \ \ \to \langle x,z \rangle \in \pole. \label{lem:real_bot:assump}
%\end{split}
\end{align}
Taking $z : = \langle 0, f , y,0\rangle$, we have, by the parts 1 and 2,
\[
	((z)_1)_0 \in \lvert (z)_0 \in \lVert \bot \rVert \rvert  \land ((z)_1)_1 \in \lVert \langle y,(z)_0 \rangle \in \pole \rVert .
\]
Thus, (\ref{lem:real_bot:assump}) yields $\langle x,z \rangle = \langle x,0,f,y,0 \rangle \in \pole$.
We now put $t := \lambda b. \langle x,0,f,y,0 \rangle$.
Then, we have $t \cdot z \simeq \langle x,0,f,y,0 \rangle \in \pole$ for all $z$.
Hence, $t$ indeed realises $\bot$ by the axiom ($\mathrm{Ax}_{\pole}$). \qed
\end{proof}

\begin{lemma}\label{lem:real_gamma}
There exists a closed term \( t \) such that
%\begin{enumerate}
%\item 
$\mathsf{FSR} \vdash t \in \lvert \exists y (y \mathrm{T} \ulcorner \gamma \urcorner) \to \gamma \rvert$.

%\item $\mathsf{FSR} \vdash s \in \lvert \gamma \rvert$.
%\end{enumerate}
\end{lemma}

\begin{proof}
%Addressing 1, 
Note that $\exists y (y \T \ulcorner \gamma \urcorner) \to \gamma$ is logically equivalent to 
\[\forall y, g \Big( y \T \ulcorner \gamma \urcorner \to \big( \forall x ( g _x \T \ulcorner \gamma \urcorner) \to \bot \big) \Big).
\]
To find a realiser for this formula, it suffices to construct a term $s$ with parameters $y,g$ such that:
\[
\mathsf{FSR} \vdash y\T\ulcorner \gamma \urcorner \to \big( \forall x ( g _x \T \ulcorner \gamma \urcorner) \to s \T \ulcorner \bot \urcorner \big).
\]
In fact, by self-realisability of $\mathsf{FSR}$ and Lemma~\ref{lem:partial_compositionality_GCR}, we can use such a term $s$ to get a term $r$ with parameters $y,g,a,b$ such that: 
\[
\mathsf{FSR} \vdash a \in \lvert y\T\ulcorner \gamma \urcorner \rvert \land b \in \lvert \forall x ( g _x \T \ulcorner \gamma \urcorner ) \rvert \to r \in \lvert s \T \ulcorner \bot \urcorner \rvert,
\]
which, by Lemma~\ref{lem:real_bot}, implies the existence of a term $q$ with parameters $y,g,a,b$ such that:
\[
\mathsf{FSR} \vdash a \in \lvert y \T \ulcorner \gamma \urcorner \rvert \land b \in \lvert \forall x ( g _x \T \ulcorner \gamma \urcorner ) \rvert \to q \in \lvert \bot \rvert,
\]
from which we can easily define a desired term that realises $\exists y(y \T \ulcorner \gamma \urcorner) \to \gamma$ by Lemma~\ref{lem:partial_compositionality_GCR}.

Now we work in $\mathsf{FSR}$.
To find a term $s$ as above, we take any $y,g$ and assume $y \T \ulcorner \gamma \urcorner $ and $\forall x ( g _x \T \ulcorner \gamma \urcorner)$.
Since $\gamma$ is equivalent to $\neg \exists g \forall x ( g _x \T \ulcorner \gamma \urcorner )$, there exists a term $f(y)$ with the parameter $y$ such that $y \T \ulcorner \gamma \urcorner $ implies the following:
\[
f(y) \T \ulcorner \neg \exists g \forall x ( g _x \T \ulcorner \gamma \urcorner ) \urcorner,
\]
which, by Lemma~\ref{lem:partial_compositionality_GCR}, yields:
\begin{align}
\forall z \big( z \T \ulcorner \exists g \forall x ( g _x \T \ulcorner \gamma \urcorner ) \urcorner \to (\mathsf{i} \cdot \langle f(y), z \rangle) \T \ulcorner \bot \urcorner \big). \label{lem:real_gamma:exist}
\end{align}
By Lemma~\ref{lem:partial_compositionality_GCR}, we have $\forall w \big( w \T \ulcorner \forall x ( g _x \T \ulcorner \gamma \urcorner ) \urcorner \to (\mathsf{e} \cdot \langle w,g \rangle) \T \ulcorner \exists g \forall x ( g _x \T \ulcorner \gamma \urcorner ) \urcorner \big)$, thus the following holds by letting $z := \mathsf{e} \cdot \langle w,g\rangle$ in (\ref{lem:real_gamma:exist}):
\begin{align}
\forall w \big( w \T \ulcorner \forall x ( g _x \T \ulcorner \gamma \urcorner ) \urcorner \to (\mathsf{i} \cdot \langle f(y), \mathsf{e} \cdot \langle w,g \rangle \rangle) \T \ulcorner \bot \urcorner \big). \label{lem:real_gamma:univ}
\end{align}
By Lemma~\ref{lem:partial_compositionality_GCR}, we have $\forall v \Big( \forall x \big( (v \cdot x) \T \ulcorner g _x \T \ulcorner \gamma \urcorner \urcorner \big) \to (\mathsf{u} \cdot v) \T \ulcorner \forall x ( g _x \T \ulcorner \gamma \urcorner ) \urcorner \Big)$;
thus, we get the following by letting $w := \mathsf{u} \cdot v$ in (\ref{lem:real_gamma:univ}):
\begin{align}
\forall v \Big( \forall x \big( (v \cdot x) \T \ulcorner g _x \T \ulcorner \gamma \urcorner \urcorner \big) \to (\mathsf{i} \cdot \langle f(y), \mathsf{e} \cdot \langle \mathsf{u} \cdot v,g \rangle \rangle) \T \ulcorner \bot \urcorner \Big). \label{lem:real_gamma:g_x+1}
\end{align}
Since $(\lambda a. g \cdot (a+1)) \cdot x \simeq g \cdot (x+1)$ is true,
we obtain the following by substituting $\lambda a. g \cdot (a+1)$ for $v$ in (\ref{lem:real_gamma:g_x+1}):
\[
\forall x (  {g_{x+1}} \T \ulcorner \gamma \urcorner ) \to (\mathsf{i} \cdot \langle f(y), \mathsf{e} \cdot \langle \mathsf{u} \cdot (\lambda a. g \cdot (a+1)),g \rangle \rangle) \T \ulcorner \bot \urcorner.
\]
By logical weakening, we also have:
\[
\forall x (  {g_{x}} \T \ulcorner \gamma \urcorner ) \to (\mathsf{i} \cdot \langle f(y), \mathsf{e} \cdot \langle \mathsf{u} \cdot (\lambda a. g \cdot (a+1)),g \rangle \rangle) \T \ulcorner \bot \urcorner,
\]
whose antecedent is one of the initial assumptions.
Therefore, we obtain the succedent.
Consequently, $s := \mathsf{i} \cdot \langle f(y), \mathsf{e} \cdot \langle \mathsf{u} \cdot (\lambda a. g \cdot (a+1)),g \rangle \rangle$ is a required term with the parameters $y,g$. \qed

%We now attend to 2.

\end{proof}

Now, we can verify Proposition~\ref{almost-omega-inconsistency}:

\begin{proof}[of Proposition~\ref{almost-omega-inconsistency}]
\begin{enumerate}
\item We work in $\mathsf{FSR}$.
By Lemma~\ref{lem:real_gamma}, there exists a realiser for $\exists y (y \T \ulcorner \gamma \urcorner) \to \gamma$.
Moreover, by the definition of $\gamma$, we already have a realiser for $\neg \gamma \to \exists y (y \T \ulcorner \gamma \urcorner)$.
Since realisability is closed under classical logic, we get a realiser for $\gamma$. 

\item By item~1, we have a term $t_0$ such that $\mathsf{FSR} \vdash t_0 \in \lvert \gamma \rvert$.
%On the other hand, we get by 
Thus, by $\mathrm{NECR}$, we also get $\mathsf{FSR} \vdash t_0 \mathrm{T} \ulcorner \gamma \urcorner$.
By self-realisability of $\mathsf{FSR}$, there exists a term $t_1$ such that $\mathsf{FSR} \vdash t_1 \in \lvert t_0 \T \ulcorner \gamma \urcorner \rvert$.
In this way, we can effectively define an infinite sequence whose initial segments sequentially realise $\gamma$.
So, we take an enumeration $g$ such that $\mathsf{FSR} \vdash g_n \T \ulcorner \gamma \urcorner$ for each numeral $n$.

\item Firstly, supposing $\pole = \emptyset$ in $\mathsf{FSR}$, we prove $\gamma$.
In the proof of Lemma~\ref{lem:real_gamma}, we get a term $s$ such that:
\[
\mathsf{FSR} \vdash y\T\ulcorner \gamma \urcorner \to \big( \forall x ( g _x \T \ulcorner \gamma \urcorner) \to s \T \ulcorner \bot \urcorner \big).
\]
In the presence of $\pole = \emptyset$, the formula $s \T \ulcorner \bot \urcorner$ is equivalent to $\bot$.
Thus, $\gamma$ follows.

Secondly, assume $\gamma$.
For a contradiction, we further suppose that $\pole \neq \emptyset$.
Then, by Lemma~\ref{lem:continuation_constant_GCR}, there exists a number $x$ which realises every $\Lr$-sentence.
We now take a constant function $g$ that always returns $x$, i.e., $\mathsf{PA} \vdash \forall y (g \cdot y \simeq x)$.
Then, $\gamma \ \big( \leftrightarrow \neg \exists g \forall y (g_y \T \ulcorner \gamma \urcorner) \big)$ implies $\exists y \neg \big(x \T (f'(g, y, \ulcorner \gamma \urcorner)) \big)$, which implies that some $\Lr$-sentence is not realised by $x$, a contradiction.
Therefore, $\gamma$ implies $\pole = \emptyset$. \qed

\end{enumerate}
\end{proof}

%%%%%%%%%%%%%%%%%%%%%%%%%%%%%%%%%%%%%%%%%%%%%%

\section{Conclusion}\label{sec:classical_real_concl}
In this paper, we have generalised the framework of compositional realisability of~\cite{hayashi2024compositional} to that of type-free realisability based on Friedman--Sheard truth~\cite{friedman1987axiomatic}.
We then observed some results on the proof-theoretic relationship between Friedman--Sheard realisability and Friedman--Sheard truth.
We list some questions concerning $\mathsf{FSR}$:

\begin{enumerate}
\item Let $\mathsf{FSR}^{++}$ be $\mathsf{FSR}$ augmented with the following strong reflection rule:
\begin{displaymath}
\infer[]{A}{ s \mathrm{T} \ulcorner A \urcorner}, \text{ for every closed term $s$ and every $\Lr$-sentence $A$}.
\end{displaymath}
Is $\mathsf{FSR}^{++}$ consistent? In particular, does it have the same theorems as $\mathsf{FSR}^{\emptyset}$?
%In particular, is $\mathsf{FSR}^{++}$ satisfied in the revision models for any choice of pole? 

\item Is $\mathsf{FS}$ relatively truth definable in $\mathsf{FSR}^{\emptyset}$? 
Similarly, is $\mathsf{FS}$ realisable in $\mathsf{FSR}$?

\item Are there any definitions of $x \in \lvert A \rvert$ and $x \in \lVert A \rVert$ and the corresponding Friedman--Sheard system such that recursion theorem is not used in the formulation while self-realisability still holds?
\end{enumerate}

We conclude this paper with two directions for future work.
Firstly, our formulation of classical realisability is based on one of Friedman and Sheard's systems.
However, given that many other self-referential approaches to truth have been developed (cf.~\cite{cantini1990theory,field2008saving,friedman1987axiomatic,kripke1976outline,leitgeb2005truth}),
it will also be natural to consider classical realisability based on those approaches.
Secondly, note that this paper and \cite{hayashi2024compositional} are mainly concerned with proof-theoretic study.
So, perhaps there may be room for further model-theoretic investigations. In particular, model-theoretic results on typed or type-free truth might be generalised to those for classical realisability.

\begin{credits}
\subsubsection{\ackname}

\end{credits}

%
% ---- Bibliography ----
%
% BibTeX users should specify bibliography style 'splncs04'.
% References will then be sorted and formatted in the correct style.
%
 \bibliographystyle{splncs04}
 \bibliography{inyou}
%
%\begin{thebibliography}{8}
%\bibitem{ref_article1}
%Author, F.: Article title. Journal \textbf{2}(5), 99--110 (2016)

%\bibitem{ref_lncs1}
%Author, F., Author, S.: Title of a proceedings paper. In: Editor,
%F., Editor, S. (eds.) CONFERENCE 2016, LNCS, vol. 9999, pp. 1--13.
%Springer, Heidelberg (2016). \doi{10.10007/1234567890}

%\bibitem{ref_book1}
%Author, F., Author, S., Author, T.: Book title. 2nd edn. Publisher,
%Location (1999)

%\bibitem{ref_proc1}
%Author, A.-B.: Contribution title. In: 9th International Proceedings
%on Proceedings, pp. 1--2. Publisher, Location (2010)

%\bibitem{ref_url1}
%LNCS Homepage, \url{http://www.springer.com/lncs}, last accessed 2023/10/25
%\end{thebibliography}

\end{document}